\documentclass[11pt]{article}

\usepackage{mathrsfs}
\usepackage{kbordermatrix}
\usepackage{enumerate}
\usepackage[section]{algorithm}
\usepackage{algorithmic}
\usepackage{multirow}
\usepackage{xcolor}

\usepackage{graphicx}
\usepackage{amsmath,amsfonts,amsthm,amsbsy,amssymb}
\usepackage{amssymb,latexsym}

\usepackage{slashbox}

\usepackage{hyperref}

\usepackage{cleveref}

\def\bu{\pmb{u}}
\def\bv{\pmb{v}}

\def\by{\pmb{y}}

\def\bbC{\mathbb{C}}
\def\bbI{\mathbb{I}}

\def\bbR{\mathbb{R}}

\def\scrG{\mathscr{G}}
\def\scrH{\mathscr{H}}

\def\scrR{\mathscr{R}}

\def\cR{{\cal R}}

\def\jbsvd{{\sc jbsvd}}
\def\jsvd{{\sc jsvd}}

\def\pjbsvd{{\sc pjbsvd}}
\def\pjsvd{{\sc pjsvd}}

\newcommand\STM[2]{{\rm St}(#1,#2)}

\def\wtd{\widetilde}
\def\what{\widehat}

\usepackage{accents}

\DeclareMathOperator{\BDiag}{BDiag}

\DeclareMathOperator{\diag}{diag}
\DeclareMathOperator{\dist}{dist}

\DeclareMathOperator*{\opt}{opt}
\DeclareMathOperator{\rank}{rank}
\DeclareMathOperator{\sym}{sym}
\DeclareMathOperator{\tr}{tr}

\DeclareMathOperator{\F}{F}
\DeclareMathOperator{\HH}{H}
\DeclareMathOperator{\T}{T}

\DeclareMathOperator{\KKT}{KKT}

\def\scrR{\mathscr{R}}

\newtheorem{theorem}{Theorem}[section]
\newtheorem{lemma}{Lemma}[section]

\theoremstyle{definition}

\newtheorem{remark}{Remark}[section]

\allowdisplaybreaks

\numberwithin{equation}{section}
\numberwithin{figure}{section}
\numberwithin{table}{section}

\def\sss{\scriptstyle}

\title{
An NPDo Approach for Principal Joint SVD-type Block Diagonalization
}

\author{Ren-Cang Li%
\thanks{Department of Mathematics, University of Texas at Arlington, Arlington, TX 76019-0408, USA.
        Supported in part by NSF DMS-2407692.
        Email: {\tt rcli@uta.edu}.}
\and
Li Wang%
\thanks{Department of Mathematics, University of Texas at Arlington, Arlington, TX 76019-0408, USA.
        Supported in part by NSF DMS-2407692.
        Email: {\tt li.wang@uta.edu}.}
\and
Mei Yang%
\thanks{Department of Mathematics, University of Texas at Arlington, Arlington, TX 76019-0408, USA.
        Email: {\tt mei.yang@uta.edu}.}
}

\date{
May 9, 2026
}

\begin{document}

\maketitle


\begin{abstract}
This paper is concerned with partial {\em Joint SVD-type Block Diagonalization\/} of several matrices
so that the extracted diagonal parts collectively optimally assume part of the total mass of all given matrices.
For that reason, it will be referred also as {\em Principal Joint SVD-type Block Diagonalization}.
When each block-size is 1-by-1, it is about finding a dominant partial joint SVD decomposition for the matrices of interests.
An NPDo approach is proposed for maximizing the common dominant block-diagonal parts collectively.
It is shown that the NPDo approach combined with Gauss-Seidel-type updating is globally convergent to a stationary point
while the objective increases monotonically.
Numerical experiments are presented to illustrate the efficiency of the NPDo approach.

\bigskip
\noindent
{\bf Keywords:}
principal joint block-diagonalization,
principal joint SVD-type block-diagonalization,
alternating SCF,
ASCF,
NPDo

\smallskip
\noindent
{\bf Mathematics Subject Classification}  62H25; 65F30; 65K05; 90C26
\end{abstract}

\clearpage
\tableofcontents

\clearpage
\section{Introduction}\label{sec:intro}
Joint SVD-type block-diagonalization (\jbsvd)
is about finding a pair of unitary matrices that can transform several given matrices to the same block-diagonal structure all at once.
For two or more matrices, this is in general cannot be done in theory because there just does not exist such a
pair, but practically one can always look for an approximate pair
that transforms the given matrices to the block-diagonal form simultaneously and optimally in certain sense.

When each block size is 1-by-1, it becomes the problem of joint SVD (\jsvd). \jsvd\ arises from
various applications such as
the estimate of two dimensional direction of arrivals, images processing, computer vision
and blind source separation \cite{copp:2010,guwe:2007,hori:2009,hori:2010,metm:2024,micc:2018,pppp:2001,sato:2015,yawl:2018}.

Specifically, let $B_{\ell}\in\bbC^{n_1\times n_2}$
for $1\le \ell\le N$ be  $N$ given matrices.
\jbsvd\ is about seeking a pair $(U,V)\in\bbC^{n_1\times n_1}\times\bbC^{n_2\times n_2}$ of unitary matrices  such that
$U^{\HH}B_{\ell}V$ for $1\le \ell\le N$ are block-diagonal with the same given block-diagonal structure.
However, in this paper, we aim at a more general setting: partial/principal joint SVD-type block-diagonalization (\pjbsvd), by which we means two things:
1) $U\in\bbC^{n_1\times k}$ and $V\in\bbC^{n_2\times k}$ have orthonormal columns, where $1\le k\le \min\{n_1,n_2\}$ but potentially $k<\min\{n_1,n_2\}$,
such that each
$U^{\HH}B_{\ell}V\in\bbC^{k\times k}$ is approximately in the prescribed block-diagonal shape, and
2) collectively the block-diagonal parts from
all $U^{\HH}B_{\ell}V\in\bbC^{k\times k}$ should pick up most of the ``weight''
in all $B_{\ell}$ as much as possible.

To explain the specifics, following \cite{cali:2019}, let
integer $k$ satisfy $1\le k\le \min\{n_1,n_2\}$, and let $\tau_k$ be a {\em partition\/} of $k$:
\begin{equation}\label{eq:tau-n}
\tau_k=(k_1,\dots,k_t),\quad
\mbox{integer $k_i\ge 1$ for $1\le i\le t$, and
     $k:=\sum\limits_{i=1}^t k_i\le \min\{n_1,n_2\}$}.
\end{equation}
For the partition $\tau_k$ and given $A\in\mathbb{R}^{k\times k}$ partitioned accordingly as
\begin{equation}\label{eq:tau-n-part}
A=\kbordermatrix{ &\sss k_1 & \sss k_2 & \sss \cdots &\sss k_t \\
         \sss k_1 & A_{11} & A_{12} & \cdots & A_{1t} \\
         \sss k_2 & A_{21} & A_{22} & \cdots & A_{2t} \\
         \sss \vdots & \vdots & \vdots &  & \vdots \\
         \sss k_t & A_{t1} & A_{t2} & \cdots & A_{tt} },
\end{equation}
we define the {\em $\tau_k$-block diagonal part\/} of $A$ as
\begin{equation*}
    \BDiag_{\tau_k}(A)=\diag(A_{11},\dots,A_{tt}).
\end{equation*}
The matrix $A$ is referred to as {\em a $\tau_k$-block diagonal matrix\/} if $A=\BDiag_{\tau_k}(A)$.
Sometimes we may simply speak of ``block diagonal'' without explicitly attaching a partition $\tau_k$ to
it, but implicitly there is one, for convenience.
Another important notation is the complex Stiefel manifold:
\begin{equation}\label{eq:cmpx-STM}
\STM{k}{n}=\{P\in\bbC^{n\times k}\,:\,P^{\HH}P=I_k\}\subset\bbC^{n\times k},
\end{equation}
the set of all $n\times k$ complex orthonormal matrices where $k\le n$.

Our goal in this paper is to seek a principal joint SVD-type block diagonalization in the sense of finding
a pair $(U,V)\in\STM{k}{n_1}\times\STM{k}{n_2}$ of orthonormal matrices  such that
$U^{\HH}B_{\ell}V$ for $1\le \ell\le N$ are optimally close to being $\tau_k$-block diagonal.
To that end, mathematically we will investigate
\begin{equation}\label{eq:pbj-intro}
\max_{(U,V)\in\STM{k}{n_1}\times\STM{k}{n_2}} \Big\{f(U,V):=\sum_{\ell=1}^N\|\BDiag_{\tau_k}(U^{\HH}B_{\ell}V)\|_F^2\Big\},
\end{equation}
and numerically
we will develop an alternating NPDo approach for its numerical solution. The term
NPDo (nonlinear polar decomposition with orthonormal polar factor dependency) was first coined in \cite{li:2024}.
For problem \eqref{eq:pbj-intro} whose objective $f(U,V)$ involves the Euclidian product of two  Stiefel manifolds, the KKT condition,  at optimality, leads to
two coupled polar decompositions of two matrix-valued functions that are dependent on their orthonormal polar factors, as we will show later.

That potentially $k<\min\{n_1,n_2\}$ makes it possible to attack large scale problems, i.e., for large
$n_1$ and/or $n_2$. This is particularly important in today's big data era where data of high dimensions is
increasingly commonplace, and often in high dimensional cases, only the dominant ``components'' are practically important and
therefore to be sought.
Hence it is critical to design numerical methods that can handle large matrices.
Our NPDo approach falls into such a category of methods. It is simple to implement and thus can be easily embedded into
large/complex applications.
We emphasize two special cases:
\begin{itemize}
  \item All $k_i=1$ and hence $t=k$: it will result in the most dominate  partial joint SVD,
        which we will refer to as principal joint SVD (\pjsvd);
  \item $t=1$ and hence $k_1=k$: it will result in the most dominate joint compression.
\end{itemize}
There were studies for the first case: all $k_i=1$ and either partial or full \jsvd.
Sato~\cite{sato:2015} investigated partial \jsvd\ and proposed to compute it
by a Riemannian trust-region method.
Congedo, Phlypo, and  Pham~\cite{copp:2010} were also interested in partial \jsvd\ of square matrices and proposed an
alternating algorithm,  drawing inspiration from the power iteration heuristically from numerical linear algebra
but without convergence analysis (see \Cref{rk:NPDo:SVD} for an explanation).

Optimization problem \eqref{eq:pbj-intro}
can in principle be solved by any general purposed optimization technique  that has been adapted from
general purposed optimization techniques to matrix manifolds over the past three decades
(see \cite{abms:2008,manopt,edas:1999,galc:2018,wagl:2021,weyi:2013} and references therein).
However, as well demonstrated recently in \cite{wazl:2022a,wazl:2023,zhys:2020,zhli:2014a,zhli:2014b,zhwb:2022}
optimization on Stiefel manifolds from various data science and engineering applications has strong
numerical linear algebra characteristics
which, if properly explored, can lead to more efficient customized algorithms than
general purposed optimization techniques.
This is what motivates us and what we will do in this paper to problem \eqref{eq:pbj-intro}.

Another inconvenience of existing general purposed optimization algorithms/packages for
optimization on matrix manifolds is that they are often designed for real matrix manifolds,
and an optimization problem
such as \eqref{eq:pbj-intro} will have to be first transformed, equivalently and properly,
into one on {\em real\/} matrix manifolds before an existing solver can be called upon to help.
Our NPDo approach to be detailed in this paper, as well as our earlier works on related topic \cite{wazl:2022a,wazl:2023,zhwb:2022},
do not have such an inconvenience.
Throughout this paper,
our presentation on \eqref{eq:pbj-intro}
is in terms of complex Stiefel manifolds \eqref{eq:cmpx-STM}, and in the case when all $B_{\ell}$ are real, one can simply limit
$U$ and $V$ to their respective real Stiefel manifolds without any change.

%
%
%

The rest of this paper is organized as follows. \Cref{sec:NPDo} reviews the NPDo approach for maximizing sum of coupled traces
\cite{li:2024,wazl:2022a}. In \cref{sec:NPDo:SVD}, we first get to the specifics of the optimization model
that determines \pjbsvd\ and then establish our alternating NPDo approach
and its LOCG acceleration for solving
the model. Detailed convergence analysis is conducted in \cref{sec:cvg-jbSVD}.
Numerical results by our NPDo approach on randomly generated testing sets with various degree of being
jointly SVD-type block-diagonalizable are reported in \cref{sec:egs-jbSVD}. Finally, we draw our conclusions
in \cref{sec:concl}.

{\bf Notation.}
We follow the following notation convention throughout this paper.
  $\bbR^{m\times n}$  is the set of $m\times n$ real matrices,  $\bbR^n=\bbR^{n\times 1}$, and $\bbR=\bbR^1$,
        and similarly $\bbC^{m\times n}$,  $\bbC^n$, and $\bbC$ except for the complex numbers.
$I_n\in\bbR^{n\times n}$ is the identity matrix or simply $I$ if its size is clear from the context.
For a matrix/vector $X$,
  $X^{\T}$ and $X^{\HH}$ stand for its transpose and complex conjugate transpose, respectively.
For $X\in\bbC^{m\times n}$, $\|X\|_2$, $\|X\|_{\F}$, and $\|X\|_{\tr}$
are its spectral norm, Frobenius norm, and trace norm (also known as the nuclear norm), respectively, and
$\sigma_{\min}(X)$ is its smallest singular
value\footnote {It is understood that $X$ has $\min\{m,n\}$
     singular values.}.
$\cR(X)$ is the column space of $X$, spanned by its columns. If also $X\in\bbC^{m\times n}$ is square, i.e., $m=n$, $\tr(X)$ is its trace and $\sym(X)=(X+X^{\HH})/2$.
A matrix $A\succ 0\, (\succeq 0)$ means that it is Hermitian and positive definite (semi-definite), and
accordingly
$A\prec 0\, (\preceq 0)$ if $-A\succ 0\, (\succeq 0)$.

\section{
         Sum of Coupled Traces}\label{sec:NPDo}
Our working engine for efficiently solving \eqref{eq:pbj-intro} is the NPDo approach for maximizing the sum of coupled traces
that was originally investigated in \cite{bomt:1998,wazl:2022a} and has since been much extended (see \cite[Example 5.1]{li:2024}).
We consider
\begin{subequations}\label{eq:OptOnSTM-master0}
\begin{equation}\label{eq:OptOnSTM-master0a}
\max_{P\in\STM{k}{n}} \Big\{ f(P):=\sum_{i=1}^M\tr\big(P_i^{\HH}A_iP_i\big)\Big\},
\end{equation}
where $\STM{k}{n}$ is the complex Stiefel manifold \eqref{eq:cmpx-STM}, and
\begin{equation}\label{eq:OptOnSTM-master0b}
      \framebox{
      \parbox{12.0cm}{
      $A_i\succeq 0$ for $1\le i\le M$, $P_i\in\bbC^{n\times k_i}$ for $1\le i\le M$ are
submatrices consisting of a few or all columns of $P$ (in particular, sharing
common columns of $P$ by different $P_i$ is allowed).
      }}
\end{equation}
\end{subequations}

\begin{remark}\label{rk:Pi}
We comment on a couple of distinctions regarding $P_i$ in \eqref{eq:OptOnSTM-master0b} from previous ones in
\cite{bomt:1998,wazl:2022a}: in \cite{bomt:1998} each $P_i$ is the $i$th column of $P$ and (thus also $M=k$)
and whereas in \cite{wazl:2022a} each $P_i$ consists of consecutive columns of $P$ and no two $P_i$ share a common column of $P$. Arguably,
 allowing two $P_i$ share a common column does not make it mathematically more general because
each trace $\tr\big(P_i^{\HH}A_iP_i\big)$ can be broken into the sum of entries each of which involves just one column of $P$
and therefore one can always first break up all traces $\tr\big(P_i^{\HH}A_iP_i\big)$ and then
combine the like terms (i.e., those involving the same column of $P$). Nonetheless, the present form
\eqref{eq:OptOnSTM-master0} can be application-friendlier sometimes. Also, such a break-up of each
$\tr\big(P_i^{\HH}A_iP_i\big)$ may no longer work for objective $f$ that is more broader than
in \eqref{eq:OptOnSTM-master0} such as some convex composition of $\tr\big(P_i^{\HH}A_iP_i\big)$ as in
\cite[Example 5.1]{li:2024} or more generally a convex composition of $\tr\big(\big[P_i^{\HH}A_iP_i\big]^s\big)$
for some integer $s\ge 1$
as in \cite{lilw:2026}.
\end{remark}

Previous studies in \cite{bomt:1998,li:2024,wazl:2022a} focused on the real number field and their extensions
to the complex number field for a real-valued function are more or less straightforward but some modifications
are necessary. In particular, we need to define properly the Euclidean gradient of the real-valued function.
For that, we will follow the brief discussion in \cite[section~2.1]{lilw:2026}:
the Euclidean gradient of the real-valued function $f$ at $P\in \bbC^{n\times k}$ is the unique matrix $\nabla f(P) \in \bbC^{n\times k}$ such that for $E\in \bbC^{n\times k}$ with sufficiently small $\|E\|_2$
\begin{equation}\label{eq:f(P+E)}
f(P+E)=f(P)+\langle \nabla f(P),E  \rangle
          +o(\|E\|_2),
\end{equation}
where $\langle X,Y  \rangle =   \Re(\tr(Y^{\HH}X))$ for $X,\, Y\in \bbC^{n\times k}$ is
the standard inner product on $\bbC^{n\times k}$, and $\Re(\cdot)$ takes the real part of a complex number.

It can be seen that each $P_i$ can be expressed as
\begin{equation}\label{eq:Pi=PJi}
P_i=PJ_i\quad\mbox{for $1\le i\le M$},
\end{equation}
where
each $J_i\in\bbR^{k\times k_i}$ is a submatrix of the identity matrix $I_k$, consisting of those columns of $I_k$
with the same column indices as $P_i$ to $P$.
 We have
\begin{equation}\label{eq:scrH-NPDo}
\scrH(P):=\nabla f(P) 
    =\sum_{i=1}^M\big[\nabla \tr\big(P_i^{\HH}A_iP_i\big)\big]J_i^{\T} 
    =\sum_{i=1}^M2A_iP_iJ_i^{\T}\in\bbC^{n\times k}.
\end{equation}
With \eqref{eq:scrH-NPDo}, the KKT condition for \eqref{eq:OptOnSTM-master0} can be stated as \cite[section~2]{li:2024}
\begin{equation}\label{eq:KKT-master0}
\scrH(P)=P\Lambda
\quad \mbox{with} \quad
\Lambda^{\HH}=\Lambda\in\bbC^{k\times k},\quad P\in\STM{k}{n}.
\end{equation}
A critical property for the objective function $f(P)$ of \eqref{eq:OptOnSTM-master0} is
\begin{equation}\label{eq:NPDo-satisfied}
\framebox{
\parbox{13cm}{\em
{\rm \cite{wazl:2022a}}
Given $P,\,\what P\in\bbC^{n\times k}$,
if
$$
\Re(\tr(\what P^{\HH}\scrH(P)))\ge\tr(P^{\HH}\scrH(P))+\eta
$$
for some $\eta\in\bbR$,
then $f(\what P)\ge f(P)+\eta$.
}}
\end{equation}
In the NPDo theory \cite{li:2024}, this confirms that  the objective function $f(P)$
satisfies {\bf the NPDo Ansatz} postulated there. Two important implications of
\eqref{eq:NPDo-satisfied} are: 1) at any maximizer $P$, \eqref{eq:KKT-master0} is satisfied
and at the same time $\Lambda\succeq 0$, i.e., it is the polar decomposition of $\scrH(\cdot)$
evaluated at $P$, and 2) the following SCF (self-consistent-field) iteration
\begin{equation}\label{eq:SCF-form:NPDo:intro}
      \framebox{
      \parbox{12.0cm}{
      SCF for NPDo \eqref{eq:KKT-master0}: given $P^{(0)}\in\STM{k}{n}$, iteratively
      compute polar decomposition
      $\scrH(P^{(j-1)})=P^{(j)}\Lambda_j$ of $\scrH(P^{(j-1)})$ for $P^{(j)}\in\STM{k}{n}$.
      }}
\end{equation}
is convergent \cite[Theorems~3.2 and 3.3]{li:2024}.
The first implication leads to the name {\em nonlinear polar decomposition with orthogonal factor dependency\/} (NPDo).
In general $k\ll n$ when $n$ is large, and the polar decomposition in SCF \eqref{eq:SCF-form:NPDo:intro}
should be computed by the thin SVD: $\scrH(P^{(j-1)})=Q_j\Sigma_jS_j^{\HH}$ where $Q_j\in\STM{k}{n}$ and
$S_j\in\STM{k}{k}$, and then $P^{(j)}=Q_jS_j^{\HH}$.

\section{Principal Joint SVD-type Block Diagonalization}\label{sec:NPDo:SVD}
Given  $\tau_k$ as in \eqref{eq:tau-n} and $N$ square matrices $B_{\ell}\in\bbC^{n_1\times n_2}$ for $1\le \ell\le N$, by
{\em principal joint SVD-type block diagonalization\/} (\pjbsvd), we mean to seek two orthonormal matrices $U\in\STM{k}{n_1}$ and $V\in\STM{k}{n_2}$ such
that $U^{\HH}B_{\ell}V$ are simultaneously and optimally close to $\BDiag_{\tau_k}(U^{\HH}B_{\ell}V)$ for all $1\le\ell\le N$ in the sense of \eqref{eq:pbj-intro}, i.e.,
\begin{equation}\label{eq:opt-pjbSVD}
\max_{U\in\STM{k}{n_1},\,V\in\STM{k}{n_2}} \left\{ f(U,V):=\sum_{\ell=1}^N\|\BDiag_{\tau_k}(U^{\HH}B_{\ell}V)\|_{\F}^2\right\},
\end{equation}
where $1\le k\le\min\{n_1,n_2\}$, and $\BDiag_{\tau_k}(\cdot)$ are defined as
in \cref{sec:intro}.
Setting all $k_i=1$ will result in the most dominate partially joint SVD, which we will call
{\em principal joint SVD-type diagonalization\/} (\pjsvd).

\subsection{The  NPDo approach -- alternating SCF}\label{ssec:NPDo-alt}
Partition $U$ and $V$ columnwise as
\begin{equation}\label{eq:P-part'n}
U=\kbordermatrix{ &\sss k_1 &\sss k_2 &\sss \cdots &\sss k_t \\
                  & U_1 & U_2 & \cdots & U_t}, \quad
V=\kbordermatrix{ &\sss k_1 &\sss k_2 &\sss \cdots &\sss k_t \\
                  & V_1 & V_2 & \cdots & V_t}
\end{equation}
or, alternatively, $U_i=UJ_i^{\T}$ and $V_i=VJ_i^{\T}$ where current $J_i$ for $1\le i\le t$ are from partitioning $I_k$ column-wise according to $\tau_k$ as
\begin{equation}\label{eq:scrH-NPDo-3}
I_k=\kbordermatrix{ &\sss k_1 &\sss k_2 &\sss \cdots &\sss k_t \\
                  & J_1 & J_2 & \cdots & J_t}.
\end{equation}
The objective $f(U,V)$ of \eqref{eq:opt-pjbSVD} can then be expressed as
\begin{subequations}\label{eq:pjbSVD:obj-tr}
\begin{align}
f(U,V)=\sum_{\ell=1}^N\sum_{i=1}^t\|U_i^{\HH}B_{\ell}V_i\|_{\F}^2
    &=\sum_{\ell=1}^N\sum_{i=1}^t\tr\big(U_i^{\HH}[B_{\ell}V_iV_i^{\HH}B_{\ell}^{\HH}]U_i\big) \label{eq:pjbSVD:obj-tr-1} \\
    &=\sum_{\ell=1}^N\sum_{i=1}^t\tr\big(V_i^{\HH}[B_{\ell}^{\HH}U_iU_i^{\HH}B_{\ell}]V_i\big). \label{eq:pjbSVD:obj-tr-2}
\end{align}
\end{subequations}
In \eqref{eq:pjbSVD:obj-tr}, there are two different reformulations of the objective $f(U,V)$
in terms of matrix traces and both take the form as
the sum of coupled traces as in \eqref{eq:OptOnSTM-master0}. They are for the purpose of
optimizing $f(U,V)$ alternatingly over $U$ and $V$, respectively.

To derive the KKT condition of \eqref{eq:opt-pjbSVD},
we note that
\begin{subequations}\label{eq:opt-pjbSVD-partdiff}
\begin{align}
\scrH_1(U,V):= 
           \nabla_Uf(U,V)
    &=2\sum_{\ell=1}^N\Big[\big(B_{\ell}V_1V_1^{\HH}B_{\ell}^{\HH}\big)U_1,\ldots,\big(B_{\ell}V_tV_t^{\HH}B_{\ell}^{\HH}\big)U_t\Big]
          \nonumber \\
    &=2\sum_{\ell=1}^N\Big[B_{\ell}V_1\big(V_1^{\HH}B_{\ell}^{\HH}U_1\big),\ldots,B_{\ell}V_t\big(V_t^{\HH}B_{\ell}^{\HH}U_t\big)\Big],
         \label{eq:opt-pjbSVD-partdiff-1} \\
\scrH_2(U,V):= 
           \nabla_Vf(U,V)
    &=2\sum_{\ell=1}^N\Big[\big(B_{\ell}^{\HH}U_1U_1^{\HH}B_{\ell}\big)V_1,\ldots,\big(B_{\ell}^{\HH}U_tU_t^{\HH}B_{\ell}\big)V_t\Big]
          \nonumber \\
    &=2\sum_{\ell=1}^N\Big[B_{\ell}^{\HH}U_1\big(U_1^{\HH}B_{\ell}V_1\big),\ldots,B_{\ell}^{\HH}U_t\big(U_t^{\HH}B_{\ell}V_t\big)\Big],
         \label{eq:opt-pjbSVD-partdiff-2}
\end{align}
\end{subequations}
yielding the KKT condition of \eqref{eq:opt-pjbSVD} as follows:
\begin{subequations}\label{eq:pjbSVD:KKT}
\begin{align}
\scrH_1(U,V)&=U\Lambda_1, \quad U\in\STM{k}{n_1}, \quad \Lambda_1=\Lambda_1^{\HH}\in\bbC^{k\times k},
         \label{eq:pjbSVD:KKT-1} \\
\scrH_2(U,V)&=V\Lambda_2, \quad V\in\STM{k}{n_2}, \quad \Lambda_2=\Lambda_2^{\HH}\in\bbC^{k\times k}.
         \label{eq:pjbSVD:KKT-2}
\end{align}
\end{subequations}
We emphasize that the reformulations in \eqref{eq:opt-pjbSVD-partdiff-1} and \eqref{eq:opt-pjbSVD-partdiff-2} of their respective expressions
at the lines before them have numerical implication, i.e., saving works because each $\big(V_i^{\HH}B_{\ell}^{\HH}U_i\big)$ is $k_i$-by-$k_i$ whereas
$\big(B_{\ell}V_iV_i^{\HH}B_{\ell}^{\HH}\big)$ is $n_i$-by-$n_i$ and
hence each block-column $\big(B_{\ell}V_iV_i^{\HH}B_{\ell}^{\HH}\big)U_i$
should be evaluated
as $B_{\ell}V_i\big(V_i^{\HH}B_{\ell}^{\HH}U_i\big)$.
It can be verified that
\begin{equation}\label{eq:recover-f}
\tr\big(U^{\HH}\scrH_1(U,V)\big)=\tr\big(V^{\HH}\scrH_2(U,V)\big)= f(U,V).
\end{equation}
As a result of \eqref{eq:NPDo-satisfied} \cite{wazl:2022a}, we have

\begin{theorem}\label{thm:f(UV)-Ansatz}
Denote by $\scrH_i$ the partial Euclidean gradients of $f(U,V)$ with respect to $U$ and $V$, respectively, as in \eqref{eq:opt-pjbSVD-partdiff},
and let $(U,V)\in\bbC^{n_1\times k}\times\bbC^{n_2\times k}$. The following statements hold.
\begin{enumerate}[{\rm (a)}]
  \item For any $\what U\in\bbC^{n_1\times k}$, if
\begin{equation}\label{eq:f(U,V)-Ansatz-1}
\Re(\tr(\what U^{\HH}\scrH_1(U,V)))\ge\tr(U^{\HH}\scrH_1(U,V))+\eta_1\quad\mbox{for some $\eta_1\in\bbR$},
\end{equation}
then $f(\what U,V)\ge f(U,V)+\eta_1$.
  \item For any $\what V\in\bbC^{n_2\times k}$, if
\begin{equation}\label{eq:f(U,V)-Ansatz-2}
\Re(\tr(\what V^{\HH}\scrH_2(U,V)))\ge\tr(V^{\HH}\scrH_2(U,V))+\eta_2\quad\mbox{for some $\eta_2\in\bbR$},
\end{equation}
then $f(U,\what V)\ge f(U,V)+\eta_2$.
\end{enumerate}
\end{theorem}

\Cref{thm:f(UV)-Ansatz} is broader than what we will need in two aspects: 1) the theorem does not require that $U,\what U, V, \what V$ have orthonormal columns, and 2) the theorem does not require $\eta_i\ge 0$.
However, if $U\in\STM{k}{n_1}$, then $\eta_1$ achieves its maximum, among $\what U\in\STM{k}{n_1}$,
at $\what U$ being the orthonormal polar factor of $\scrH_1(U,V)$ because
$\Re(\tr(\what U^{\HH}\scrH_1(U,V)))$ attains its maximum at such $\what U$ \cite{li:2026}, and the similar thing can be said about $\eta_2$. This fact forms the foundation of
\Cref{alg:NPDo-pjbSVD}.

Along the line of the proof of \cite[Theorem~3.1]{li:2024}, we can also get

\begin{theorem}\label{thm:maximizer-pjSVD}
Let $(U_*,V_*)\in\STM{k}{n_1}\times\STM{k}{n_2}$ be a maximizer pair of \eqref{eq:opt-pjbSVD}.
Then the KKT condition
\eqref{eq:pjbSVD:KKT} holds with
$(U,V)=(U_*,V_*)$, i.e.,
\begin{subequations}\label{eq:pjbSVD:KKTatMAX}
\begin{align}
\scrH_1(U_*,V_*)&=U_*\Lambda_{1*}, \quad \Lambda_{1*}=U_*^{\HH}\scrH_1(U_*,V_*)\succeq 0,
         \label{eq:pjbSVD:KKTatMAX-1} \\
\scrH_2(U_*,V_*)&=V_*\Lambda_{2*}, \quad \Lambda_{2*}=V_*^{\HH}\scrH_2(U_*,V_*)\succeq 0.
         \label{eq:pjbSVD:KKTatMAX-2}
\end{align}
\end{subequations}
\end{theorem}

\begin{algorithm}[t]
\caption{NPDoJBSVD: the NPDo approach for solving \eqref{eq:opt-pjbSVD}.}
\label{alg:NPDo-pjbSVD}
\begin{algorithmic}[1]
\REQUIRE $B_{\ell}\in\bbC^{n_1\times n_2}$ for $1\le \ell\le N$
         (and, accordingly, matrix-valued functions $\scrH_i(\cdot,\cdot)$ in \eqref{eq:opt-pjbSVD-partdiff}),
         and initial $(U^{(0)},V^{(0)})\in\STM{k}{n_1}\times \STM{k}{n_2}$;
\ENSURE  an approximate maximizer pair of \eqref{eq:opt-pjbSVD}.
\FOR{$j=0,1,\ldots$ until convergence}
    \STATE for Jacobi-type updating only:
           \begin{itemize}
             \item compute $U^{(j+1)}$ and $V^{(j+1)}$ as the orthonormal polar factors of $\scrH_1(U^{(j)},V^{(j)})$ and $\scrH_2(U^{(j)},V^{(j)})$, respectively;
           \end{itemize}
    \STATE for Gauss-Seidel-type updating only:
           \begin{itemize}
             \item compute $U^{(j+1)}$ as the orthonormal polar factor of $\scrH_1(U^{(j)},V^{(j)})\in\bbC^{n_1\times k}$, and then
                   compute $V^{(j+1)}$ as the orthonormal polar factor of $\scrH_2(U^{(j+1)},V^{(j)})\in\bbC^{n_2\times k}$;
           \end{itemize}
\ENDFOR
\RETURN  the last $(U^{(j)},V^{(j)})$.
\end{algorithmic}
\end{algorithm}

\Cref{thm:maximizer-pjSVD} says, at any maximizer pair, each of the equations in \eqref{eq:pjbSVD:KKT} is a polar decomposition.
With \Cref{thm:f(UV)-Ansatz,thm:maximizer-pjSVD}, naturally we may solve
\eqref{eq:opt-pjbSVD} via maximizing $f(U,V)$ alternatingly over $U$ and $V$.
We outline the procedure in \Cref{alg:NPDo-pjbSVD},
which includes two updating schemes -- Jacobi-type or Gauss-Seidel-type.
The two types of updating schemes mimic the ones commonly used in linear system solving \cite{demm:1997}.
Considering they falls into the well-known self-consistent-field (SCF) iteration style, we will refer
them as {\em alternating SCF\/} (ASCF) iterations.
A couple of comments about implementing \Cref{alg:NPDo-pjbSVD} are in order.
\begin{enumerate}[(1)]
  \item \Cref{alg:NPDo-pjbSVD} has two variants in one: with Jacobi-type updating or with Gauss-Seidel-type updating.
  \item At Lines 2 and 3, the thin SVD provides an efficient way to compute the orthonormal polar factors for very small $k\ll\min\{n_1,n_2\}$.
  \item A reasonable stopping criterion at Line 1 is
        \begin{equation}\label{eq:stop-pjbSVD}
        \varepsilon_{\KKT}:=\frac {\big\|\scrH_1(U,V)-U\Lambda_1(U,V)\big\|_{\F}+\big\|\scrH_2(U,V)-V\Lambda_2(U,V)\big\|_{\F}}
                {2\sum_{\ell=1}^N\|B_{\ell}\|_{\F}\|B_{\ell}\|_2}
              \le\epsilon,
        \end{equation}
        where $\Lambda_1(U,V)=\sym\big(U^{\HH}\scrH_1(U,V)\big)$ and $\Lambda_2(U,V)=\sym\big(V^{\HH}\scrH_2(U,V)\big)$, and
        $\epsilon$ is a given tolerance.
        For programming consideration, it is more efficient to evaluate $\varepsilon_{\KKT}$
        at $(U^{(j)},V^{(j)})$, the pair of the most current approximations for Jacobi-type updating
        and, for Gauss-Seidel-type updating,  to evaluate $\scrH_1(U,V)-U\Lambda_1(U,V)$ at $(U^{(j)},V^{(j-1)})$
        while $\scrH_2(U,V)-V\Lambda_2(U,V)$ at $(U^{(j)},V^{(j)})$.


        Computing $\|B_{\ell}\|_2$ is nontrivial for large $n_1$ and $n_2$, but fortunately for  normalization purpose some rough estimate
        is good enough, e.g., replacing $\|B_{\ell}\|_2$ with $\sqrt{\|B_{\ell}\|_1\|B_{\ell}\|_{\infty}}$.
        Another possibility is to use the Golub-Kahan-Lanczos bidiagonalization \cite{bddrv:2000} and often
        running a few bidiagonalization steps  can produce a very good estimate of $\|B_{\ell}\|_2$, for the same reason as in \cite{zhli:2011}.
\end{enumerate}
Finally with \Cref{thm:f(UV)-Ansatz}, for Gauss-Seidel-type updating, the general convergence theorems, similar to \cite[Theorems~3.2~and~3.3]{li:2024},
can be obtained. It is a little bit more complicated for the Jacobi-type updating. We will spell out the detail in
\cref{sec:cvg-jbSVD}.

\begin{remark}\label{rk:NPDo:SVD}
For the case of \jsvd,  the alternating scheme in \cite{copp:2010} to maximize $f(U,V)$ looks similar
to ours (for the Guass-Seidel-type updating). In fact, one of the variations is the same as our
\Cref{alg:NPDo-pjbSVD} specialized to \jsvd.
To make it clear that each $U_i$ and $V_i$ are now vectors, we write $\bu_i$ and $\bv_i$ instead and rewrite $f(U,V)$ as
\begin{subequations}\label{eq:objv-rk}
\begin{align}
f(U,V)&=\sum_{i=1}^k\tr\Big(\bu_i^{\HH}\Big[\sum_{\ell=1}^NB_{\ell}\bv_i\bv_i^{\HH}B_{\ell}^{\HH}\Big]\bu_i\Big) \label{eq:objv-rk-1}\\
      &=\sum_{i=1}^k\tr\Big(\bv_i^{\HH}\Big[\sum_{\ell=1}^NB_{\ell}^{\HH}\bu_i\bu_i^{\HH}B_{\ell}\Big]\bv_i\Big). \label{eq:objv-rk-2}
\end{align}
\end{subequations}
Each reformulation, \eqref{eq:objv-rk-1} or \eqref{eq:objv-rk-2}, is the sum of $k$ traces, and the alternating scheme in \cite{copp:2010}
attempts to maximize $f(U,V)$ as if those $k$ traces were independent.
Basically, each iterative step of the alternating procedure in
\cite[Algorithm~1]{copp:2010}, computing the next approximations $\{\tilde\bu_i\}_{i=1}^k$ and $\{\tilde\bv_i\}_{i=1}^k$ given current approximations $\{\bu_i\}_{i=1}^k$ and $\{\bv_i\}_{i=1}^k$,
is as follows:
\begin{enumerate}[1.]
  \item compute
      \begin{equation}\label{eq:power-u}
      \hat\bu_i=\Big[\sum_{\ell=1}^NB_{\ell}\bv_i\bv_i^{\HH}B_{\ell}^{\HH}\Big]\bu_i
      \quad\mbox{for $1\le i\le k$},
      \end{equation}
      which can be viewed as one step of the usual power iteration \cite{demm:1997}, and then
      orthogonalize $\{\hat\bu_i\}_{i=1}^k$ to yield $\{\tilde\bu_i\}_{i=1}^k$
      as the new approximations;
  \item compute 
      \begin{equation}\label{eq:power-v}
      \hat\bv_i=\Big[\sum_{\ell=1}^NB_{\ell}^{\HH}\tilde\bu_i\tilde\bu_i^{\HH}B_{\ell}\Big]\bv_i
      \quad\mbox{for $1\le i\le k$},
      \end{equation}
      which again can be viewed as one step of the usual power iteration, and then orthogonalize $\{\hat\bv_i\}_{i=1}^k$ to yield $\{\tilde\bv_i\}_{i=1}^k$
      as
      the new approximations.
\end{enumerate}
Heuristically, theose $k$ one-step power iterations in \eqref{eq:power-u} or \eqref{eq:power-v}
attempt to improve each $\bu$-vector or $\bv$-vector and they are done
as if the summands in either \eqref{eq:objv-rk-1} or \eqref{eq:objv-rk-2} {\em independently}.
As a result, $\hat\bu_i$ for $1\le i\le n$
are unlikely orthogonal to each other and so are $\hat\bv_i$ for $1\le i\le n$.
In both step 1 and step 2 there is an orthogonalization actions.
The authors in \cite{copp:2010} attempted to fix the non-orthogonality issue by performing re-orthogonalization.
They experimented with the Gram-Schmit orthogonalization as well as the L{\"o}wdin orthogonalization that
is not very well known even today in the numerical linear algebra community!
The Gram-Schmit orthogonalization is flawed because
the process degrades any benefit from the one-step power iterations, i.e., the objective value
at the new approximations may potentially be less than at the old ones.
The L{\"o}wdin orthogonalization, discovered (in a non-SVD form), is named after a Swedish chemist, Per-Olov L{\"o}wdin, for the purpose of orthogonalizing electon orbitals. In the current context, the L{\"o}wdin orthogonalization is simply to compute
the orthonormal polar factor of $[\hat\bu_1,\ldots,\hat\bu_k]$ or that of
$[\hat\bv_1,\ldots,\hat\bv_k]$. It turns out that is the same as the orthonormal polar factor of
$\scrH_1(U,V)$ or $\scrH_2(\what U,V)$, which we do at Line 3 of \Cref{alg:NPDo-pjbSVD}.
The authors reported in \cite{copp:2010} that ``L{\"o}wdin's method has proved more robust.''
Given what we now know from  \Cref{thm:f(UV)-Ansatz}, that conclusion of theirs is hardly any surprising at all.
But it is not clear from \cite{copp:2010} if the authors were aware of,   at that time, some form of \Cref{thm:f(UV)-Ansatz}.
%
%
%
%
\end{remark}

\subsection{Acceleration with LOCG}\label{ssec:accNPDo}
We can create a LOCG-accelerated version of \Cref{alg:NPDo-pjbSVD}.
Without loss of generality, let $(U^{(-1)},V^{(-1)})$ be the pair of approximate maximizers of \eqref{eq:opt-pjbSVD}
from the very previous iterative step, and $(U^{(0)},V^{(0)})$ the  pair of current approximate maximizers.
We are now looking for the next approximate maximizer
$(U^{(1)},V^{(1)})$, along the line of LOCG (locally optimal conjugate gradient), according to
\begin{subequations}\label{eq:LOCG}
\begin{align}
&(U^{(1)},V^{(1)})=\arg\max_{X\in\STM{k}{n_1},Y\in\STM{k}{n_2}}f(X,Y) \label{eq:LOCG-1}\\
&\mbox{s.t.}\,\, \cR(X)\subseteq\cR([U^{(0)},\scrR_1(U^{(0)},V^{(0)}),U^{(-1)}]), \label{eq:LOCG-2}\\
&\hphantom{\mbox{s.t.}\,\,}\cR(Y)\subseteq\cR([V^{(0)},\scrR_2(U^{(0)},V^{(0)}),V^{(-1)}]),
             \label{eq:LOCG-3}
\end{align}
\end{subequations}
where $\cR(\cdot)$ denotes the range of a matrix, , i.e., the subspace spanned by the columns of the matrix, and
\begin{subequations}\label{eq:Rs(U,V)}
\begin{align}
\scrR_1(U,V)&=\scrH_1(U,V)-U\sym\big(U^{\HH}\scrH_1(U,V)\big), \label{eq:Rs(U,V)-1}\\
\scrR_2(U,V)&=\scrH_2(U,V)-V\sym\big(V^{\HH}\scrH_2(U,V)\big). \label{eq:Rs(U,V)-2}
\end{align}
\end{subequations}
Initially for the first iteration, we don't have $(U^{(-1)},V^{(-1)})$ and
it is understood that $U^{(-1)}$ and $V^{(-1)}$ are absent from \eqref{eq:LOCG-2} and \eqref{eq:LOCG-3}, i.e.,
simply
$\cR(X)\subseteq\cR([U^{(0)},\scrR_1(U^{(0)},V^{(0)})])$ and $\cR(Y)\subseteq\cR([V^{(0)},\scrR_2(U^{(0)},V^{(0)})])$.

We still have to numerically solve \eqref{eq:LOCG}. For that purpose, let $S_i\in\STM{\hat n_i}{n_i}$ be an orthonormal basis matrix of subspace
$\cR([U^{(0)},\scrR_1(U^{(0)},V^{(0)}),U^{(-1)}])$ and $\cR([V^{(0)},\scrR_2(U^{(0)},V^{(0)}),V^{(-1)}])$, respectively. Generically, $\hat n_i=3k$ but $\hat n_i<3k$ can happen.
It can be implemented by the Gram-Schmidt orthogonalization process. For example, for $\cR([U^{(0)},\scrR_1(U^{(0)},V^{(0)}),U^{(-1)}])$,
we notice that $U^{(0)}\in\STM{k}{n_1}$ already. In MATLAB, to fully take advantage of its optimized functions, we simply set
$S_1=[\scrR_1(U^{(0)},V^{(0)}),U^{(-1)}]$ (or $S_1=\scrR_1(U^{(0)},V^{(0)})$ for the first iteration) and then  do
\begin{equation}\label{eq:W-compute}
\framebox{
\begin{minipage}{6.5cm}
\tt S1=S1-U0*(U0'*S1); S1=orth(S1); \\
\tt S1=S1-U0*(U0'*S1); S1=orth(S1);\\
\tt S1=[U0,S1];
\end{minipage}
}
\end{equation}
where the first two lines  perform the classical Gram-Schmidt orthogonalization twice to almost ensure that
the resulting  columns of $S_1$ are fully orthogonal to the columns of $U^{(0)}$ at the end of the second line,
and {\tt orth} is a MATLAB function for
orthogonalization\footnote{
    Another option is to use MATLAB's thin {\tt qr}:
    {\tt [S1,$\sim$]=qr(S1,0)}. }.
It is important to note that the first $k$ columns of
the final $S_1$ are the same as those of $U^{(0)}$. Similarly, we compute $S_2\in\STM{\hat n_2}{n_2}$ such that
the first $k$ columns of $S_2$ are $V^{(0)}$. So we get
\begin{subequations}\label{eq:S1S2:LOCG}
\begin{align}
\cR([U^{(0)},\scrR_1(U^{(0)},V^{(0)}),U^{(-1)}])=\cR(S_1), \label{eq:S1:LOCG}\\
\cR([V^{(0)},\scrR_2(U^{(0)},V^{(0)}),V^{(-1)}])=\cR(S_2). \label{eq:S2:LOCG}
\end{align}
\end{subequations}
Hence \eqref{eq:LOCG-2} and \eqref{eq:LOCG-3} can be equivalently stated as
\begin{subequations}\label{eq:LOCGsub}
\begin{equation}\label{eq:LOCGsub:Y}
X=S_1Z_1\quad\mbox{for $Z_1\in\STM{k}{\hat n_1}$}, \quad
Y=S_2Z_2\quad\mbox{for $Z_2\in\STM{k}{\hat n_2}$}.
\end{equation}
Problem \eqref{eq:LOCG} becomes
\begin{equation}\label{eq:LOCGsub-1}
(Z_{1;\opt},Z_{2;\opt})=\arg\max_{Z_1\in\STM{k}{\hat n_1},\,Z_2\in\STM{k}{\hat n_2}}\, \wtd f(Z_1,Z_2),
\end{equation}
where, upon setting $\wtd B_{\ell}=S_1^{\HH}B_{\ell}S_2$,
\begin{equation}\label{eq:obj-eq:LOCGsub}
\wtd f(Z_1,Z_2):=f(S_1Z_1,S_2Z_2)
   =\sum_{\ell=1}^N\|\BDiag_{\tau_k}(Z_1^{\HH}\wtd B_{\ell}Z_2)\|_{\F}^2,
\end{equation}
in the same form of \eqref{eq:opt-pjbSVD} but with much smaller $\wtd B_{\ell}$.
\end{subequations}

\begin{lemma}\label{lm:LOCGreduced}
We have for the reduced problem~\eqref{eq:LOCGsub-1} with objective \eqref{eq:obj-eq:LOCGsub}
\begin{subequations}\label{eq:partD-relation}
\begin{align}
\wtd\scrH_1(Z_1,Z_2):=\nabla_{Z_1}\wtd f(Z_1,Z_2)=S_1^{\HH}\scrH_1(S_1Z_1,S_2Z_2), \label{eq:partD-relation-1} \\
\wtd\scrH_2(Z_1,Z_2):=\nabla_{Z_2}\wtd f(Z_1,Z_2)=S_2^{\HH}\scrH_2(S_1Z_1,S_2Z_2), \label{eq:partD-relation-2}
\end{align}
\end{subequations}
and, with $Z_1^{(0)}$ and $Z_2^{(0)}$ being the first $k$ columns of $I_{\hat n_1}$ and $I_{\hat n_2}$ respectively,
\begin{subequations}\label{eq:init-relation}
\begin{align}
\|\wtd\scrH_1(Z_1^{(0)},Z_2^{(0)})\|_{\tr}&=\|\scrH_1(U^{(0)},V^{(0)})\|_{\tr}, \label{eq:init-relation-1a}\\
\tr\big([Z_1^{(0)}]^{\HH}\wtd\scrH_1(Z_1^{(0)},Z_2^{(0)})\big)&=\tr\big([U^{(0)}]^{\HH}\scrH_1(U^{(0)},V^{(0)})\big),
            \label{eq:init-relation-1b} \\
\|\wtd\scrH_2(Z_1^{(0)},Z_2^{(0)})\|_{\tr}&=\|\scrH_2(U^{(0)},V^{(0)})\|_{\tr}, \label{eq:init-relation-2a}\\
\tr\big([Z_2^{(0)}]^{\HH}\wtd\scrH_2(Z_1^{(0)},Z_2^{(0)})\big)&=\tr\big([V^{(0)}]^{\HH}\scrH_2(U^{(0)},V^{(0)})\big).
            \label{eq:init-relation-2b}
\end{align}
\end{subequations}
\end{lemma}

\begin{proof}
Notice that $\wtd f(Z_1,Z_2)=f(S_1Z_1,S_2Z_2)$. According to the definition \eqref{eq:f(P+E)} of the Euclidean
gradient of a real-valued function, we have
\begin{align*}
\wtd f(Z_1+E,Z_2)&=\wtd f(Z_1,Z_2)+\langle\nabla_{Z_1}\wtd f(Z_1,Z_2),E\rangle+o(\|E\|_2), \\
\wtd f(Z_1+E,Z_2)&=f(S_1[Z_1+E],S_2Z_2)=f(S_1Z_1+S_1E,S_2Z_2) \\
  &=f(S_1Z_1,S_2Z_2)+\langle\left.\nabla_U f(U,V)\right|_{(U,V)=(S_1Z_1,S_2Z_2)},S_1E\rangle+o(\|E\|_2) \\
  &=\wtd f(Z_1,Z_2)+\langle\left.S_1^{\HH}\nabla_U f(U,V)\right|_{(U,V)=(S_1Z_1,S_2Z_2)},E\rangle+o(\|E\|_2),
\end{align*}
yielding
$$
\nabla_{Z_1}\wtd f(Z_1,Z_2)=\left.S_1^{\HH}\nabla_U f(U,V)\right|_{(U,V)=(S_1Z_1,S_2Z_2)}
   =S_1^{\HH}\scrH_1(S_1Z_1,S_2Z_2),
$$
which is \eqref{eq:partD-relation-1}. Similarly, we get \eqref{eq:partD-relation-2}.
By the way how $S_1$ and $S_2$ are computed, it can be seen that
$S_1Z_1^{(0)}=U^{(0)}$ and $S_2Z_2^{(0)}=V^{(0)}$. It follows from \eqref{eq:partD-relation-1} that
$\wtd\scrH_1(Z_1^{(0)},Z_2^{(0)})=S_1^{\HH}\scrH_1(U^{(0)},V^{(0)})$.
Immediately,
$$
[Z_1^{(0)}]^{\HH}\wtd\scrH_1(Z_1^{(0)},Z_2^{(0)})
  =[Z_1^{(0)}]^{\HH}S_1^{\HH}\scrH_1(S_1Z_1^{(0)},S_2Z_2^{(0)})
  =[U^{(0)}]^{\HH}\scrH_1(U^{(0)},V^{(0)}),
$$
leading to \eqref{eq:init-relation-1b}.
With the help of \eqref{eq:Rs(U,V)-1} and \eqref{eq:S1:LOCG}, we have
\begin{align*}
S_1S_1^{\HH}\scrH_1(U^{(0)},V^{(0)})
   &=S_1S_1^{\HH}\big[\scrR_1(U^{(0)},V^{(0)})+U^{(0)}\sym\big([U^{(0)}]^{\HH}\scrH_1(U^{(0)},V^{(0)})\big)\big]\\
   &=S_1S_1^{\HH}\scrR_1(U^{(0)},V^{(0)})+S_1S_1^{\HH}U^{(0)}\sym\big([U^{(0)}]^{\HH}\scrH_1(U^{(0)},V^{(0)})\big)\\
   &=\scrR_1(U^{(0)},V^{(0)})+U^{(0)}\sym\big([U^{(0)}]^{\HH}\scrH_1(U^{(0)},V^{(0)})\big) \\
   &=\scrH_1(U^{(0)},V^{(0)}),
\end{align*}
yielding
\begin{align*}
\|\wtd\scrH_1(Z_1^{(0)},Z_2^{(0)})\|_{\tr}
  &=\|S_1^{\HH}\scrH_1(U^{(0)},V^{(0)})\|_{\tr} \\
  &=\|S_1S_1^{\HH}\scrH_1(U^{(0)},V^{(0)})\|_{\tr}
  =\|\scrH_1(U^{(0)},V^{(0)})\|_{\tr},
\end{align*}
which is \eqref{eq:init-relation-1a}. The equalities \eqref{eq:init-relation-2a} and \eqref{eq:init-relation-2b}
can be verified in the same way.
\end{proof}

It is noted that each $\wtd B_{\ell}$ is of $\hat n_1\times \hat n_2$ and $\hat n_i\ll n_i$ for $i=1,2$ if $k\ll \{n_1,n_2\}/3$.
Solving the reduced problem~\eqref{eq:LOCGsub-1} by \Cref{alg:NPDo-pjbSVD} should be much faster.
We caution that, within \Cref{alg:NPDo-pjbSVD} for solving \eqref{eq:LOCGsub-1},
$\wtd\scrH_i(\cdot,\cdot)$ should be evaluated straight from $\wtd B_{\ell}$, not by the relationships in \eqref{eq:partD-relation} we just established.
\Cref{alg:accNPDo-pjbSVD} outlines what we have discussed so far.

\begin{algorithm}[t]
\caption{The LOCG-accelerated NPDo for solving \eqref{eq:opt-pjbSVD}}
\label{alg:accNPDo-pjbSVD}
\begin{algorithmic}[1]
\REQUIRE $B_{\ell}\in\bbC^{n_1\times n_2}$ for $1\le \ell\le N$
         (and, accordingly, matrix-valued functions $\scrH_i(\cdot,\cdot)$ for $i=1,2$ as in \eqref{eq:opt-pjbSVD-partdiff}),
         and initial $(U^{(0)},V^{(0)})\in\STM{k}{n_1}\times \STM{k}{n_2}$;
\ENSURE  an approximate maximizer pair of \eqref{eq:opt-pjbSVD}.
\STATE $U^{(-1)}=[\,]$, $V^{(-1)}=[\,]$; \% null matrices
\FOR{$j=0,1,\ldots$ until convergence}
    \STATE compute $R_i^{(j)}=\scrR_i(U^{(j)},V^{(j)})$ for $i=1,2$, where
           $\scrR_i(U^{(j)},V^{(j)})$ is calculated according to \eqref{eq:Rs(U,V)};
    \STATE compute $S_i\in\STM{\hat n_i}{n_i}$ for $i=1,2$ such that
           $\cR(S_1)=\cR([U^{(j)},R_1^{(j)},U^{(j-1)}])$
           and
           $\cR(S_2)=\cR([V^{(j)},R_2^{(j)},V^{(j-1)}])$
           as in \eqref{eq:W-compute} to ensure that the first $k$ columns of $S_1$ are the same as those of $U^{(j)}$
           and the first $k$ columns of $S_2$ are the same as those of $V^{(j)}$;
    \STATE solve \eqref{eq:LOCGsub-1} for $(Z_{1;\opt},Z_{2;\opt})$ by \Cref{alg:NPDo-pjbSVD}
          with initially $Z_i^{(0)}$
           being the first $k$ columns of $I_{\hat n_i}$;
    \STATE $U^{(j+1)}=S_1Z_{1;\opt}$, $V^{(j+1)}=S_2Z_{2;\opt}$;
\ENDFOR
\RETURN last $(U^{(j)},V^{(j)})$.
\end{algorithmic}
\end{algorithm}

\begin{remark}\label{rk:SCF4npd+LOCG}
There are a few  comments in order, regarding \Cref{alg:accNPDo-pjbSVD}.
\begin{enumerate}[(i)]
  \item Inheriting from \Cref{alg:NPDo-pjbSVD}, \Cref{alg:accNPDo-pjbSVD} also
        has two variants in one: with Jacobi-type updating or with Gauss-Seidel-type updating.
  \item The stopping criterion \eqref{eq:stop-pjbSVD} can be used at Line 2;
  \item It is important to compute $S_1$ at Line~4 in such a way, as explained moments ago, that its first $k$
        columns are exactly the same as those of $U^{(j)}$.
        This is because as $U^{(j)}$ converges, $U^{(j+1)}$ changes little from $U^{(j)}$ and hence $Z_{1;\opt}$
        is increasingly close to the first $k$ columns of $I_{\hat n_1}$.
        The same can be said about $S_2$. This explains the choice of $Z_i^{(0)}$
        at Line~5.
  \item At Line 4, some saving can be achieved by reusing qualities that are already computed. For example, we may use
        $$
        B_{\ell}V^{(j+1)}=(B_{\ell}S_2)Z_{2;\opt}, \quad
        [U^{(j+1)}]^{\HH}B_{\ell}V^{(j+1)}=Z_{1;\opt}^{\HH}(S_1^{\HH}B_{\ell}S_2)Z_{2;\opt}
        $$
        to compute
        the next $B_{\ell}V^{(j+1)}$ and $[U^{(j+1)}]^{\HH}B_{\ell}V^{(j+1)}$ at  costs $O(n_1k^2)$ and $O(k^3)$, respectively,
        instead of $O(n_1n_2k)$ by reusing $(B_{\ell}S_2)$ and $(S_1^{\HH}B_{\ell}S_2)$.
  \item An area of improvement is to solve the reduced problem \eqref{eq:LOCGsub-1} with an accuracy, comparable but fractionally better than the
        current $(U^{(j)},V^{(j)})$ as an approximate solution of \eqref{eq:opt-pjbSVD}.
        Specifically, if we use
        \eqref{eq:stop-pjbSVD} at Line~2 here to stop the for-loop: Lines 2--7, with tolerance $\epsilon$, then instead of using the same
        $\epsilon$ for \Cref{alg:NPDo-pjbSVD} at its Line 1 when the algorithm is called here at Line 5,
        we can use a fraction, say $1/8$,
        of $\varepsilon_{\KKT}$ evaluated at the current approximation $(U^{(j)},V^{(j)})$ as stopping tolerance within \Cref{alg:NPDo-pjbSVD}. This is what we do in our numerical experiment.
\end{enumerate}
\end{remark}

\section{Convergence Analysis}\label{sec:cvg-jbSVD}
In this section, we will perform a  convergence analysis for Algorithm~\ref{alg:NPDo-pjbSVD}.
For the sake of presentation, we introduce, for $i=1,2$,
\begin{equation}\label{eq:scrH4proof}
\scrH_i^{(j)}=\scrH_i(U^{(j)},V^{(j)})\,\,\mbox{for $i=1,2$, and}\quad
\what\scrH_2^{(j)}=\scrH_2(U^{(j+1)},V^{(j)}).
\end{equation}
Also
$\Theta(\cdot,\cdot)$ is the diagonal matrix of the canonical angles between two subspaces of equal dimension (see, e.g., \cite[appendix~A]{li:2024}, \cite{li:2026}, \cite{stsu:1990}).

\subsection{With Gauss-Seidel-type updating}
The following lemma is an equivalent restatement of \cite[Lemma 4.10]{moso:1983}
(see also \cite[Proposition 7]{kaqi:1999}) in the context of a metric space.
We will need it to prove one of the conclusions in our main theorem in this subsection, \Cref{thm:cvg4SCF4NPDo-GS:JSVD}.

\begin{lemma}[{\cite[Lemma 4.10]{moso:1983}}]\label{lm:isolatedconvg}
Let $\scrG$ be a metric space with metric $\dist(\cdot,\cdot)$, and let
$\{\by_i\}_{i=0}^{\infty}$ be a sequence in $\scrG$. If
$\by_*\in \scrG$ is an isolated accumulation point 	
of the sequence such that, for every subsequence $\{\by_i\}_{i\in\bbI}$
converging to $\by_*$, there is an infinite subset $\widehat{\bbI}\subseteq \bbI$ satisfying
$\dist(\by_i,\by_{i+1})\to 0$ as $\what\bbI\ni i\to\infty$,
then the entire sequence $\{\by_i\}_{i=0}^{\infty}$ converges to $\by_*$.
\end{lemma}

Both \Cref{thm:cvg4SCF4NPDo-GS:JSVD,thm:cvg4SCF4accNPDo-GS:JSVD} bear similarity to
the corresponding convergence theorems for the NPDo theory in \cite{li:2024} and its updated version at arXiv.

\begin{theorem}\label{thm:cvg4SCF4NPDo-GS:JSVD}
Let the sequence $\{(U^{(j)},V^{(j)})\}_{j=0}^{\infty}$ be generated by \Cref{alg:NPDo-pjbSVD}
with Gauss-Seidel-type updating.
The following statements hold.
\begin{enumerate}[{\rm (a)}]
  \item The sequence $\{f(U^{(j)},V^{(j)})\}_{j=0}^{\infty}$ is monotonically increasing and convergent;
  \item Let $(U_*,V_*)$ be an accumulation point
        of the sequence $\{(U^{(j)},V^{(j)})\}_{j=0}^{\infty}$, and let
        $\{(U^{(j)},V^{(j)})\}_{j\in\bbI}$ be a convergent subsequence that converges to $(U_*,V_*)$.
        We have
        \begin{subequations}\label{eq:pjbSVD:KKTatMAX'-GS}
        \begin{align}
        \scrH_1(U_*,V_*)&=U_*\Lambda_{1*}, \quad \Lambda_{1*}=U_*^{\HH}\scrH_1(U_*,V_*)\succeq 0,
                 \label{eq:pjbSVD:KKTatMAX'-GS-1} \\
        \scrH_2(\what U_*,V_*)&=V_*\Lambda_{2*}, \quad \Lambda_{2*}=V_*^{\HH}\scrH_2(\what U_*,V_*)\succeq 0,
                 \label{eq:pjbSVD:KKTatMAX'-GS-2}
        \end{align}
        \end{subequations}
        where $\what U_*$ is any accumulation point of $\{U^{(j+1)}\}_{j\in\bbI}$.
        As a result, \eqref{eq:pjbSVD:KKTatMAX} holds if also $\what U_*=U_*$.
  \item In item~{\rm (b)}, if $\rank(\scrH_1(U_*,V_*))=k$, then $\what U_*=U_*$ and hence \eqref{eq:pjbSVD:KKTatMAX} holds.
  \item If $(U_*,V_*)$ is an isolated accumulation point of $\{(U^{(j)},V^{(j)})\}_{j=0}^{\infty}$
        and if
        \begin{equation}\label{eq:full-rank3}
        \rank(\scrH_i(U_*,V_*))=k\quad
         \mbox{for $i=1,2$},
        \end{equation}
        then the entire sequence $\{(U^{(j)},V^{(j)})\}_{j=0}^{\infty}$ converges to $(U_*,V_*)$.
  \item We have four convergent series
        \begin{subequations}\label{eq:cvg4SCF4NPDo-GS:JSVD:series}
        \begin{align}
        \sum_{j=0}^{\infty}\sigma_{\min}(\scrH_1^{(j)})\,
                         \big\|\sin\Theta\big(\cR(U^{(j+1)}),\cR(U^{(j)})\big)\big\|_{\F}^2
                      &<\infty,    \label{eq:cvg4SCF4NPDo-GS:JSVD:series-1} \\
        \sum_{j=0}^{\infty}\sigma_{\min}(\what\scrH_2^{(j)})\,
                         \big\|\sin\Theta\big(\cR(V^{(j+1)}),\cR(V^{(j)})\big)\big\|_{\F}^2
                      &<\infty,    \label{eq:cvg4SCF4NPDo-GS:JSVD:series-1a} \\
        \sum_{j=0}^{\infty}\sigma_{\min}(\scrH_1^{(j)})\,
                  \frac {\big\|\scrH_1^{(j)}-U^{(j)}\big([U^{(j)}]^{\HH}\scrH_1^{(j)}\big)\big\|_{\F}^2}
                        {\big\|\scrH_1^{(j)}\big\|_{\F}^2}
                      &<\infty,                 \label{eq:cvg4SCF4NPDo-GS:JSVD:series-2} \\
        \sum_{j=0}^{\infty}\sigma_{\min}(\what\scrH_2^{(j)})\,
                  \frac {\big\|\what\scrH_2^{(j)}-V^{(j)}\big([V^{(j)}]^{\HH}\what\scrH_2^{(j)}\big)\big\|_{\F}^2}
                        {\big\|\what\scrH_2^{(j)}\big\|_{\F}^2}
                      &<\infty.                 \label{eq:cvg4SCF4NPDo-GS:JSVD:series-2a}
        \end{align}
        \end{subequations}
\end{enumerate}
\end{theorem}

\begin{proof}
With  Gauss-Seidel-type updating, the flow of computation goes as follows:
for $j=0,1,2,\ldots$
\begin{equation}\label{eq:flow-GS}
    (U^{(j)},V^{(j)})
\to (U^{(j+1)},V^{(j)})
\to (U^{(j+1)},V^{(j+1)})  .
\end{equation}
Recall \eqref{eq:scrH4proof} and let
\begin{subequations}\label{eq:cvg4SCF4NPDo-GS:JSVD:pf-1}
\begin{align}
\eta_{j+1/2}&=\tr\big([U^{(j+1)}]^{\HH}\scrH_1^{(j)}\big)-\tr\big([U^{(j)}]^{\HH}\scrH_1^{(j)}\big)
            \label{eq:cvg4SCF4NPDo-GS:JSVD:pf-1aa}\\
          &=\|\scrH_1^{(j)}\|_{\tr}-\tr\big([U^{(j)}]^{\HH}\scrH_1^{(j)}\big), \label{eq:cvg4SCF4NPDo-GS:JSVD:pf-1ab}\\
\eta_{j+1}&=\tr\big([V^{(j+1)}]^{\HH}\what\scrH_2^{(j)}\big)-\tr\big([V^{(j)}]^{\HH}\what\scrH_2^{(j)}\big)
            \label{eq:cvg4SCF4NPDo-GS:JSVD:pf-1ba}\\
          &=\|\what\scrH_2^{(j)}\|_{\tr}-\tr\big([V^{(j)}]^{\HH}\what\scrH_2^{(j)}\big),
            \label{eq:cvg4SCF4NPDo-GS:JSVD:pf-1bb}
\end{align}
\end{subequations}
where \eqref{eq:cvg4SCF4NPDo-GS:JSVD:pf-1ab} and \eqref{eq:cvg4SCF4NPDo-GS:JSVD:pf-1bb} are due to how
$U^{(j+1)}$ and $V^{(j+1)}$ are defined in \Cref{alg:NPDo-pjbSVD}. Both
$\eta_{j+1/2}$ and $\eta_{j+1}$ are nonnegative by \cite[Lemma~4.2]{li:2026}. Now use
\Cref{thm:f(UV)-Ansatz} to conclude
\begin{subequations}\label{eq:cvg4SCF4NPDo-GS:JSVD:pf-2'}
\begin{align}
f(U^{(j+1)},V^{(j)})&\ge f(U^{(j)},V^{(j)})+ \eta_{j+1/2}, \label{eq:cvg4SCF4NPDo-GS:JSVD:pf-2'a}\\
f(U^{(j+1)},V^{(j+1)})&\ge f(U^{(j+1)},V^{(j)})+ \eta_{j+1}, \label{eq:cvg4SCF4NPDo-GS:JSVD:pf-2'b}
\end{align}
\end{subequations}
yielding $f(U^{(j+1)},V^{(j+1)})\ge f(U^{(j+1)},V^{(j)})\ge f(U^{(j)},V^{(j)})$. This proves item (a).
Along the way, we also showed
\begin{align}
f(U^{(j+1)},V^{(j+1)})\ge f(U^{(j)},V^{(j)})
   &+\Big[\|\what\scrH_2^{(j)}\|_{\tr}-\tr\big([V^{(j)}]^{\HH}\what\scrH_2^{(j)}\big)\Big] \nonumber \\
   &+\Big[\|\scrH_1^{(j)}\|_{\tr}-\tr\big([U^{(j)}]^{\HH}\scrH_1^{(j)}\big)\Big].
        \label{eq:cvg4SCF4NPDo-GS:JSVD:pf-3}
\end{align}

We now prove item~(b).
Without loss of generality, we may  assume that $\{U^{(j+1)}\}_{j\in\bbI}$ converges to $\what U_*$; otherwise,
we can pick up a convergent subsequence of $\{U^{(j+1)}\}_{j\in\bbI}$ and reassign $\bbI$ according to the convergent subsequence.
We have
\begin{equation}\label{eq:cvg4SCF4NPDo-GS:JSVD:pf-5}
\lim_{\bbI\ni j\to \infty}U^{(j)}=U_*, \quad
\lim_{\bbI\ni j\to \infty}V^{(j)}=V_*, \quad
\lim_{\bbI\ni j\to \infty}U^{(j+1)}=\what U_*.
\end{equation}
Notice that the contrary to \eqref{eq:pjbSVD:KKTatMAX'-GS-1} is
\begin{subequations}\label{eq:cvg4SCF4NPDo-GS:JSVD:pf-6}
\begin{equation}\label{eq:cvg4SCF4NPDo-GS:JSVD:pf-6a}
\mbox{either $\cR(\scrH_1(U_*,V_*))\not\subseteq\cR(U_*)$ or
$U_*^{\HH}\scrH_1(U_*,V_*)\not\succeq 0$},
\end{equation}
and the contrary to \eqref{eq:pjbSVD:KKTatMAX'-GS-2} is
\begin{equation}\label{eq:cvg4SCF4NPDo-GS:JSVD:pf-6b}
\mbox{either $\cR(\scrH_2(\what U_*,V_*))\not\subseteq\cR(V_*)$ or
$V_*^{\HH}\scrH_2(\what U_*,V_*)\not\succeq 0$}.
\end{equation}
\end{subequations}
Hence the contrary to \eqref{eq:pjbSVD:KKTatMAX'-GS} is that at least one of the four relations in \eqref{eq:cvg4SCF4NPDo-GS:JSVD:pf-6} is true, which, by \cite[Lemma~4.2]{li:2026}, implies
that either $\delta_1>0$ or $\delta_2>0$ where
\begin{subequations}\label{eq:cvg4SCF4NPDo-GS:JSVD:pf-deltas}
\begin{align}
\delta_1&:=\|\scrH_1(U_*,V_*)\|_{\tr}-\tr(U_*^{\HH}\scrH_1(U_*,V_*))\ge 0, \label{eq:cvg4SCF4NPDo-GS:JSVD:pf-deltas-1}\\
\delta_2&:=\|\scrH_2(\what U_*,V_*)\|_{\tr}-\tr(V_*^{\HH}\scrH_2(\what U_*,V_*))\ge 0, \label{eq:cvg4SCF4NPDo-GS:JSVD:pf-deltas-2}
\end{align}
\end{subequations}
and therefore $\delta:=\delta_1+\delta_2>0$.
Since $\|\scrH_i(U,V)\|_{\tr}$ for $i=1,2$, $\tr(U^{\HH}\scrH_1(U,V))$,
$\tr(V^{\HH}\scrH_2(U,V))$, and $f(U,V)$ are continuous in $(U,V)\in\bbC^{n_1\times k}\times\bbC^{n_2\times k}$,
there is an $j_0\in\bbI$ such that
\begin{subequations}\label{eq:cvg4SCF4NPDo-GS:JSVD:pf-7}
\begin{gather}
\Big|\|\scrH_1(U_*,V_*)\|_{\tr}-\|\scrH_1^{(j_0)}\|_{\tr}\Big| <\delta/5, \label{eq:thm:cvg4SCF4NPDo-GS:pf-7a}\\
\Big|\tr((U^{(j_0)})^{\HH}\scrH_1^{(j_0)})-\tr(U_*^{\HH}\scrH_1(U_*,V_*))\Big|<\delta/5, \label{eq:thm:cvg4SCF4NPDo-GS:pf-7b}\\
\Big|\|\scrH_2(\what U_*,V_*)\|_{\tr}-\|\what\scrH_2^{(j_0)}\|_{\tr}\Big| <\delta/5, \label{eq:thm:cvg4SCF4NPDo-GS:pf-7c}\\
\Big|\tr((V^{(j_0)})^{\HH}\scrH_2^{(j_0)})-\tr(V_*^{\HH}\scrH_2(\what U_*,V_*))\Big|<\delta/5, \label{eq:thm:cvg4SCF4NPDo-GS-subs-1:pf-7d}\\
f(U_*,V_*)-\delta/10<f(U^{(j_0)},V^{(j_0)})\le f(U_*,V_*).  \label{eq:thm:cvg4SCF4NPDo-GS-subs-1:pf-7e}
\end{gather}
\end{subequations}
By \eqref{eq:cvg4SCF4NPDo-GS:JSVD:pf-3}, we have
\begin{align*}
f(U^{(j_0+1)},V^{(j_0+1)})&\ge f(U^{(j_0)},V^{(j_0)})
       +\Big[\|\what\scrH_2^{(j_0)}\|_{\tr}-\tr\big([V^{(j_0)}]^{\HH}\what\scrH_2^{(j_0)}\big)\Big]  \\
  &\hphantom{\ge f(U^{(j_0)},V^{(j_0)})}\,
       +\Big[\|\scrH_1^{(j_0)}\|_{\tr}-\tr\big([U^{(j_0)}]^{\HH}\scrH_1^{(j_0)}\big)\Big] \\
  &>f(U_*,V_*)-\frac {\delta}{10}
        +\Big[\|\scrH_2(\what U_*,V_*)\|_{\tr}-\frac {\delta}5-\tr(V_*^{\HH}\scrH_2(\what U_*,V_*))-\frac {\delta}5\Big] \\
  &\hphantom{>f(U_*,V_*)-\frac {\delta}{10}}\,
        +\Big[\|\scrH_1(U_*,V_*)\|_{\tr}-\frac {\delta}5-\tr(U_*^{\HH}\scrH_1(U_*,V_*))-\frac {\delta}5\Big] \\
  &=f(U_*,V_*)+\frac {\delta}{10}>f(U_*,V_*),
\end{align*}
contradicting $f(U^{(i)},V^{(i)})\le\lim_{j\to\infty}f(U^{(j)},V^{(j)})= f(U_*,V_*)$ for all $i$.
This completes the proof of item~(b).

Since we always have $\scrH_1(U^{(j)},V^{(j)})=U^{(j+1)}\Lambda_1^{(j)}$, a polar decomposition. Letting $\bbI\ni j\to\infty$
yields $\scrH_1(U_*,V_*)=\what U_*\what \Lambda_{1*}$, a polar decomposition also.
If $\rank(\scrH_1(U_*,V_*))=k$, then the polar decomposition of $\scrH_1(U_*,V_*)$ is
unique \cite{high:2008,li:1993b,li:1995,li:2014HLA} and hence $\what U_*=U_*$ by the fact that
\eqref{eq:pjbSVD:KKTatMAX'-GS-1} is another polar decomposition of $\scrH_1(U_*,V_*)$,
completing the proof of item (c).

We now prove item~(d) with the help of \Cref{lm:isolatedconvg}. Let
$\{(U^{(j)},V^{(j)})\}_{j\in\bbI_1}$ be a convergent subsequence that converges to $(U_*,V_*)$.
Since $\{(U^{(j+1)},V^{(j+1)})\}_{j\in\bbI_1}$ is a bounded sequence, it has a convergent subsequence
$\{(U^{(j+1)},V^{(j+1)})\}_{j\in\bbI_2}$ that converges to $(\what U_*,\what V_*)$,
where $\bbI_2\subseteq\bbI_1$.
Use the same argument for proving item (c) to conclude $\what U_*=U_*$.
Always $\scrH_2(U^{(j+1)},V^{(j)})=V^{(j+1)}\Lambda_2^{(j)}$, a polar decomposition.
Letting $\bbI_1\supseteq\bbI_2\ni j\to\infty$ yields
$\scrH_2(U_*,V_*)=\what V_*\Lambda_{2*}$ where we have used  $\what U_*=U_*$. This is yet another polar decomposition of $\scrH_2(U_*,V_*)$, besides
\eqref{eq:pjbSVD:KKTatMAX'-GS-2}. It follows from \eqref{eq:full-rank3} that
$\scrH_2(U_*,V_*)$ has a unique polar decomposition, implying $\what V_*=V_*$.
Hence as $\bbI_2\ni j\to\infty$
$$
\|U^{(j)}-U^{(j+1)}\|_{\F}\le\|U^{(j)}-U_*\|_{\F}+\|U_*-U^{(j+1)}\|_{\F}\to 0,
$$
and similarly $\|V^{(j)}-V^{(j+1)}\|_{\F}\to 0$.
Now use Lemma~\ref{lm:isolatedconvg} to conclude that the entire sequence
$\{(U^{(j)},V^{(j)})\}_{i=0}^{\infty}$ converges to $(U_*,V_*)$, as needed.

Finally for item (e), we return to \eqref{eq:cvg4SCF4NPDo-GS:JSVD:pf-3}.
By \cite[Theorem~4.1]{li:2026}, we have
\begin{subequations}\label{eq:cvg4SCF4NPDo-GS:JSVD:pf-8}
\begin{gather}
\frac {\big\|\scrH_1^{(j)}-U^{(j)}\big([U^{(j)}]^{\HH}\scrH_1^{(j)}\big)\big\|_{\F}^2}
                        {\big\|\scrH_1^{(j)}\big\|_{\F}^2}
    \le\big\|\sin\Theta\big(\cR(U^{(j+1)}),\cR(U^{(j)})\big)\big\|_{\F}^2
    \le\frac {2\eta_{j+1/2}}{\sigma_{\min}(\scrH_1^{(j)})}, \label{eq:cvg4SCF4NPDo-GS:JSVD:pf-8a}\\
\frac {\big\|\what\scrH_2^{(j)}-V^{(j)}\big([V^{(j)}]^{\HH}\what\scrH_2^{(j)}\big)\big\|_{\F}^2}
                        {\big\|\what\scrH_2^{(j)}\big\|_{\F}^2}
    \le\big\|\sin\Theta\big(\cR(V^{(j+1)}),\cR(V^{(j)})\big)\big\|_{\F}^2
    \le\frac {2\eta_{j+1}}{\sigma_{\min}(\what\scrH_2^{(j)})}. \label{eq:cvg4SCF4NPDo-GS:JSVD:pf-8b}
\end{gather}
\end{subequations}
It
follows from \eqref{eq:cvg4SCF4NPDo-GS:JSVD:pf-2'} that
$$
2\sum_{j=0}^{\infty}(\eta_{j+1/2}+\eta_{j+1})\le\lim_{j\to\infty} f(U^{(j+1)},V^{(j+1)})-f(U^{(0)},V^{(0)})<\infty,
$$
combining which with \eqref{eq:cvg4SCF4NPDo-GS:JSVD:pf-8} yields \eqref{eq:cvg4SCF4NPDo-GS:JSVD:series}. The proof is completed.
\end{proof}

As a corollary of Theorem~\ref{thm:cvg4SCF4NPDo-GS:JSVD}(e),
if $\sigma_{\min}(\scrH_1^{(j)})$ is eventually bounded below away from $0$
uniformly\footnote {By which we mean that there exist a constant $c>0$ and an integer $K$ such that
  $\sigma_{\min}(\scrH_1^{(j)})\ge c$ for all $j\ge K$.},
then
\begin{subequations}\label{eq:NPDo-always}
\begin{equation}\label{eq:eq:NPDo-always-1}
\lim_{j\to\infty}\frac {\big\|\scrH_1^{(j)}-U^{(j)}\big([U^{(j)}]^{\HH}\scrH_1^{(j)}\big)\big\|_{\F}}
                        {\big\|\scrH_1^{(j)}\big\|_{\F}} =0,
\end{equation}
namely, increasingly $\scrH_1(U^{(j)},V^{(j)})\approx U^{(j)}\big([U^{(j)}]^{\HH}\scrH_1(U^{(j)},V^{(j)})\big)$ towards
a polar decomposition of $\scrH_1(U^{(j)},V^{(j)})$, which means that $(U^{(j)},V^{(j)})$ becomes a more and more accurate approximate solution
to NPDo \eqref{eq:pjbSVD:KKT-1}, even in the absence of knowing whether the entire sequence $\{(U^{(j)},V^{(j)})\}_{j=0}^{\infty}$ converges or not. For the same reason, if $\sigma_{\min}(\what\scrH_2^{(j)})$ is eventually bounded below away from $0$
uniformly, then
\begin{equation}\label{eq:NPDo-always-2}
\lim_{j\to\infty}
\frac {\big\|\what\scrH_2^{(j)}-V^{(j)}\big([V^{(j)}]^{\HH}\what\scrH_2^{(j)}\big)\big\|_{\F}}
                        {\big\|\what\scrH_2^{(j)}\big\|_{\F}}=0,
\end{equation}
\end{subequations}
namely, increasingly $\scrH_2(U^{(j+1)},V^{(j)})\approx V^{(j)}\big([V^{(j)}]^{\HH}\scrH_2(U^{(j+1)},V^{(j)})\big)$,
regardless whether the entire sequence $\{(U^{(j+1)},V^{(j)})\}_{j=0}^{\infty}$ converges or not.

Next we will analyze the convergence of \Cref{alg:accNPDo-pjbSVD} that employs an inner and outer iterative scheme.
Each loop of the outer iteration produces a reduced problem~\eqref{eq:LOCGsub} that has to be solved
by the SCF iteration, the inner iteration,  of \Cref{alg:NPDo-pjbSVD}. In theory, each reduced problem can be solved
to full convergence as guaranteed by \Cref{thm:cvg4SCF4NPDo-GS:JSVD}, but that is not practical and wasteful, too, as
we argued in \Cref{rk:SCF4npd+LOCG}(v). For our analysis to go through, we minimally require that, at Line 5 when \Cref{alg:NPDo-pjbSVD} is called, it
produces next approximate pair $(U^{(j+1)},V^{(j+1)})$ that satisfies
\begin{multline}\label{eq:SCF4accNPDo-GS:progress}
f(U^{(j+1)},V^{(j+1)})-f(U^{(j)},V^{(j)}) \\
   \ge c\,\max\Big\{\underbrace{\|\scrH_1^{(j)}\|_{\tr}-\tr([U^{(j)}]^{\HH}\scrH_1^{(j)})}_{=:\eta_{*1}^{(j)}},
        \underbrace{\|\scrH_2^{(j)}\|_{\tr}-\tr([V^{(j)}]^{\HH}\scrH_2^{(j)})}_{=:\eta_{*2}^{(j)}}\Big\},
\end{multline}
where $c>0$ is a constant, independent of $j$. To explain why this is somewhat a minimal requirement.
Consider one NPDo SCF iterative step on
$\wtd \scrH_1(Z_1^{(0)},Z_2^{(0)})$ of the reduced problem~\eqref{eq:LOCGsub} at the entry to \Cref{alg:NPDo-pjbSVD}.
By \Cref{thm:f(UV)-Ansatz}(a), it will
increase $\wtd f$ by
\begin{align*}
\wtd f(Z_1^{(1)},Z_2^{(0)})-\wtd f(Z_1^{(0)},Z_2^{(0)})
   &\ge\|\wtd\scrH_1(Z_1^{(0)},Z_2^{(0)})\|_{\tr}-\tr\big([Z_1^{(0)}]^{\HH}\wtd\scrH_1(Z_1^{(0)},Z_2^{(0)})\big) \\
   &=\|\scrH_1(U^{(0)},V^{(0)})\|_{\tr}-\tr\big([U^{(0)}]^{\HH}\scrH_1(U^{(0)},V^{(0)})\big)=\eta_{1*}^{(j)},
\end{align*}
where we have used \Cref{lm:LOCGreduced}. Similarly, one NPDo SCF iterative step on
$\wtd \scrH_2(Z_1^{(0)},Z_2^{(0)})$ at the entry will
increase $\wtd f$ by
$$
\wtd f(Z_1^{(0)},Z_2^{(1)})-\wtd f(Z_1^{(0)},Z_2^{(0)})\ge \eta_{2*}^{(j)}.
$$
Therefore the assumption
\eqref{eq:SCF4accNPDo-GS:progress} merely asks \Cref{alg:NPDo-pjbSVD} to
solve each reduced problem~\eqref{eq:LOCGsub} about fractionally as good as the better one between the one NPDo SCF iterative step
on $\wtd \scrH_1(Z_1^{(0)},Z_2^{(0)})$ and on $\wtd \scrH_2(Z_1^{(0)},Z_2^{(0)})$, respectively.

\begin{theorem}\label{thm:cvg4SCF4accNPDo-GS:JSVD}
Let the sequence $\{(U^{(j)},V^{(j)})\}_{j=0}^{\infty}$ be generated by \Cref{alg:accNPDo-pjbSVD}
with Gauss-Seidel-type updating. Suppose that, at Line 5, \Cref{alg:NPDo-pjbSVD} delivers
the next approximate pair $(U^{(j+1)},V^{(j+1)})$ that satisfies \eqref{eq:SCF4accNPDo-GS:progress}.
Then the following statements hold.
\begin{enumerate}[{\rm (a)}]
  \item The sequence $\{f(U^{(j)},V^{(j)})\}_{j=0}^{\infty}$ is monotonically increasing and convergent;
  \item Let $(U_*,V_*)$ be an accumulation point
        of the sequence $\{(U^{(j)},V^{(j)})\}_{j=0}^{\infty}$, and let
        $\{(U^{(j)},V^{(j)})\}_{j\in\bbI}$ be a convergent subsequence that converges to $(U_*,V_*)$.
        We have \eqref{eq:pjbSVD:KKTatMAX}.
  \item We have four convergent series
        \begin{subequations}\label{eq:cvg4SCF4accNPDo-GS:JSVD:series}
        \begin{align}
        \sum_{j=0}^{\infty}\sigma_{\min}(\scrH_1^{(j)})\,
                         \big\|\sin\Theta\big(\cR(\scrH_1^{(j)}),\cR(U^{(j)})\big)\big\|_{\F}^2
                      &<\infty,    \label{eq:cvg4SCF4accNPDo-GS:JSVD:series-1} \\
        \sum_{j=0}^{\infty}\sigma_{\min}(\scrH_2^{(j)})\,
                         \big\|\sin\Theta\big(\cR(\scrH_2^{(j)}),\cR(V^{(j)})\big)\big\|_{\F}^2
                      &<\infty,    \label{eq:cvg4SCF4accNPDo-GS:JSVD:series-1a} \\
        \sum_{j=0}^{\infty}\sigma_{\min}(\scrH_1^{(j)})\,
                  \frac {\big\|\scrH_1^{(j)}-U^{(j)}\big([U^{(j)}]^{\HH}\scrH_1^{(j)}\big)\big\|_{\F}^2}
                        {\big\|\scrH_1^{(j)}\big\|_{\F}^2}
                      &<\infty,                 \label{eq:cvg4SCF4accNPDo-GS:JSVD:series-2} \\
        \sum_{j=0}^{\infty}\sigma_{\min}(\scrH_2^{(j)})\,
                  \frac {\big\|\scrH_2^{(j)}-V^{(j)}\big([V^{(j)}]^{\HH}\scrH_2^{(j)}\big)\big\|_{\F}^2}
                        {\big\|\scrH_2^{(j)}\big\|_{\F}^2}
                      &<\infty.                 \label{eq:cvg4SCF4accNPDo-GS:JSVD:series-2a}
        \end{align}
        \end{subequations}
\end{enumerate}
\end{theorem}

\begin{proof}
Item (a) is an easy consequence of the assumption
\eqref{eq:SCF4accNPDo-GS:progress}. Some minor modifications to the proof of \Cref{thm:cvg4SCF4NPDo-GS:JSVD}(b)
work, too, to prove item (b) here, e.g., use $\delta=\max\{\delta_1,\delta_2\}$ with $\delta_i$ as in \eqref{eq:cvg4SCF4NPDo-GS:JSVD:pf-deltas}.
Some minor modifications to the proof of \Cref{thm:cvg4SCF4NPDo-GS:JSVD}(e) will show item (c) here, noticing
the major difference between \eqref{eq:cvg4SCF4NPDo-GS:JSVD:series-1} and \eqref{eq:cvg4SCF4accNPDo-GS:JSVD:series-1}
and that between \eqref{eq:cvg4SCF4NPDo-GS:JSVD:series-1a} and \eqref{eq:cvg4SCF4accNPDo-GS:JSVD:series-1a}:
$\cR(U^{(j+1)})$ {\em vs.} $\cR(\scrH_1^{(j)})$ and
$\cR(V^{(j+1)})$ {\em vs.} $\cR(\scrH_2^{(j)})$.
Also the hat to $\scrH_2^{(j)}$ is absent.
\end{proof}

\subsection{With Jacobi-type updating}
With Jacobi-type updating, we can no longer claim that the sequence $\{f(U^{(j)},V^{(j)})\}_{i=0}^{\infty}$ is monotonically increasing.
To overcome the difficulty  brought by that, we will consider the sequence of   pairs of approximate maximizers
that include all intermediate pairs $(U^{(j+1)},V^{(j)})$ and $(U^{(j)},V^{(j+1)})$, besides
$(U^{(j)},V^{(j)})$. For ease of presentation, we introduce the sequence
$\{P_j\}_{j=0}^{\infty}$ defined by
\begin{equation}\label{eq:SCF4NPDo-Jac:JSVD:all}
P_{3j}=(U^{(j)},V^{(j)}), \,\,
P_{3j+1}=(U^{(j+1)},V^{(j)}), \,\,
P_{3j+2}=(U^{(j)},V^{(j+1)})\quad\mbox{for $j=0,1,2,\ldots$}.
\end{equation}

\begin{theorem}\label{thm:cvg4SCF4NPDo-Jac:JSVD}
Let the sequence $\{P_j\}_{j=0}^{\infty}$ defined by \eqref{eq:SCF4NPDo-Jac:JSVD:all}  be generated by \Cref{alg:NPDo-pjbSVD}
with Jacobi-type updating.
The following statements hold.
\begin{enumerate}[{\rm (a)}]
  \item There exists a convergent subsequence $\{P_j\}_{j\in\bbI}$ of $\{P_j\}_{j=0}^{\infty}$,
        say convergent to $P_*\equiv(U_*,V_*)$,  such that
        \begin{equation}\label{eq:SCF4NPDo-Jac:JSVD:objv}
        \limsup_{j\to\infty}f(P_j)=\lim_{\bbI\ni j\to\infty}f(P_j)=f(U_*,V_*);
        \end{equation}
  \item
        If
        $\{P_j\}_{j\in\bbI}$ in item {\rm (a)} contains infinitely many pairs of type $(U^{(j)},V^{(j)})$, then
        \eqref{eq:pjbSVD:KKTatMAX} holds.
\end{enumerate}
\end{theorem}

\begin{proof}
Since objective function $f$ is bounded above, the limit superior of $\{f(P_j)\}_{j=1}^{\infty}$ exists and
there exists a convergent subsequence $\{f(P_j)\}_{j=\bbI'}$ of $\{f(P_j)\}_{j=1}^{\infty}$ such that
$$
\limsup_{j\to\infty}f(P_j)=\lim_{\bbI'\ni j\to\infty}f(P_j).
$$
Since $\{P_j\}_{j\in\bbI'}\in\STM{k}{n_1}\times\STM{k}{n_2}$, it too has a convergent subsequence, say
$\{P_j\}_{j\in\bbI}$, which satisfies \eqref{eq:SCF4NPDo-Jac:JSVD:objv}. This proves item (a).

For item (b), we will prove \eqref{eq:pjbSVD:KKTatMAX-1} only because essentially the same argument
based on \eqref{eq:cvg4SCF4NPDo-Jac:JSVD:pf-3b} below
will lead to \eqref{eq:pjbSVD:KKTatMAX-2}. With Jacobi-type updating, we will have
\begin{subequations}\label{eq:cvg4SCF4NPDo-Jac:JSVD:pf-3}
\begin{align}
f(U^{(j+1)},V^{(j)})&\ge f(U^{(j)},V^{(j)})
       +\Big[\|\scrH_1^{(j)}\|_{\tr}-\tr\big([U^{(j)}]^{\HH}\scrH_1^{(j)}\big)\Big], \label{eq:cvg4SCF4NPDo-Jac:JSVD:pf-3a}\\
f(U^{(j)},V^{(j+1)})&\ge f(U^{(j)},V^{(j)})
       +\Big[\|\scrH_2^{(j)}\|_{\tr}-\tr\big([V^{(j)}]^{\HH}\scrH_2^{(j)}\big)\Big]. \label{eq:cvg4SCF4NPDo-Jac:JSVD:pf-3b}
\end{align}
\end{subequations}
Suppose
that \eqref{eq:pjbSVD:KKTatMAX-1} did not hold. Then, by \cite[Lemma~4.2]{li:2026},
$$
\delta:=\|\scrH_1(U_*,V_*)\|_{\tr}-\tr\big(U_*^{\HH}\scrH_1(U_*,V_*)\big)>0.
$$
Since $\|\scrH_1(U,V)\|_{\tr}$, $\tr(U^{\HH}\scrH_1(U,V))$, and $f(U,V)$ are continuous
in $(U,V)\in\bbC^{n_1\times k}\times\bbC^{n_2\times k}$, we have
for infinitely many and sufficiently large $j\in\bbI$ such that
$(U^{(j)},V^{(j)})$ is contained in $\{P_j\}_{j\in\bbI}$ and
such that
\begin{subequations}\label{eq:cvg4SCF4NPDo-Jac:JSVD:pf-5}
\begin{gather}
\Big|\|\scrH_1(U_*,V_*)\|_{\tr}-\|\scrH_1^{(j)}\|_{\tr}\Big| <\delta/3, \label{eq:thm:cvg4SCF4NPDo-Jac-subs-1:pf-5a}\\
\Big|\tr((U^{(j)})^{\HH}\scrH_1^{(j)})-\tr(U_*^{\HH}\scrH_1(U_*,V_*))\Big|<\delta/3, \label{eq:thm:cvg4SCF4NPDo-Jac-subs-1:pf-5b}\\
f(U_*,V_*)-\delta/6<f(U^{(j)},V^{(j)})\le f(U_*,V_*).  \label{eq:thm:cvg4SCF4NPDo-Jac-subs-1:pf-5ec}
\end{gather}
\end{subequations}
By \eqref{eq:cvg4SCF4NPDo-Jac:JSVD:pf-3a}, we have
\begin{align*}
f(U^{(j+1)},V^{(j)})&\ge f(U^{(j)},V^{(j)})
       +\Big[\|\scrH_1^{(j)}\|_{\tr}-\tr\big([V^{(j)}]^{\HH}\scrH_1^{(j)}\big)\Big]  \\
  &>f(U_*,V_*)-\frac {\delta}{6}
        +\Big[\|\scrH_1(U_*,V_*)\|_{\tr}-\frac {\delta}3-\tr(U_*^{\HH}\scrH_1(U_*,V_*))-\frac {\delta}3\Big] \\
  &=f(U_*,V_*)+\frac {\delta}{6}>f(U_*,V_*),
\end{align*}
which holds for infinitely many $j$ and thus
$$
\limsup_{j\to\infty}f(P_j)\ge f(U_*,V_*)+\frac {\delta}{6}>f(U_*,V_*),
$$
contradicting \eqref{eq:SCF4NPDo-Jac:JSVD:objv}. Hence \eqref{eq:pjbSVD:KKTatMAX-1} holds.
\end{proof}

\subsection{Gauss-Seidel-type {\em vs.} Jacobi-type in convergence}
There are notable differences between \Cref{thm:cvg4SCF4NPDo-GS:JSVD,thm:cvg4SCF4NPDo-Jac:JSVD}:
\begin{enumerate}
  \item With Gauss-Seidel-type updating in \Cref{thm:cvg4SCF4NPDo-GS:JSVD}, sequence $\{f(U^{(j)},V^{(j)})\}_{j=0}^{\infty}$ is
        monotonically increasing and thus convergent. However,
        with Jacobi-type updating in \Cref{thm:cvg4SCF4NPDo-Jac:JSVD},
        whether $\{f(U^{(j)},V^{(j)})\}_{j=0}^{\infty}$ is
        monotonically increasing is questionable. To circumvent that,
        we take objective values at, besides $(U^{(j)},V^{(j)})$,  all intermediate pairs of approximate maximizers
        into consideration and focus on the limit superior of  objective values at those pairs.
  \item With Gauss-Seidel-type updating in \Cref{thm:cvg4SCF4NPDo-GS:JSVD}, we do not have \eqref{eq:pjbSVD:KKTatMAX-2} but \eqref{eq:pjbSVD:KKTatMAX'-GS-2}
        instead. Only upon assuming $\rank(\scrH_1(U_*,V_*))=k$, we can claim \eqref{eq:pjbSVD:KKTatMAX} in whole.
        However, with Jacobi-type updating in \Cref{thm:cvg4SCF4NPDo-Jac:JSVD} which take all
        pairs of approximate maximizers into consideration, we get \eqref{eq:pjbSVD:KKTatMAX} without any condition on $\rank(\scrH_1(U_*,V_*))$, but note that there is a condition imposed on the subsequence
        in \Cref{thm:cvg4SCF4NPDo-Jac:JSVD}(b) that gives the
        limit superior there.
  \item With Gauss-Seidel-type updating in \Cref{thm:cvg4SCF4NPDo-GS:JSVD}, we have four informative convergent series,
        which implies \eqref{eq:NPDo-always}, provided $\sigma_{\min}(\scrH_1^{(j)})$ and $\sigma_{\min}(\what\scrH_2^{(j)})$
        are eventually bounded below away from $0$ uniformly.
        However, with Jacobi-type updating in \Cref{thm:cvg4SCF4NPDo-Jac:JSVD}, we have not found one yet.
\end{enumerate}

\subsection{Extensions}
It is interesting to observe that in proving theorems in this section the actual forms of
both $\scrH_i(\cdot,\cdot)$ play no roles at all. What have been used in the proofs are:
\begin{enumerate}[i)]
  \item $\scrH_i(\cdot,\cdot)$ for $i=1,2$ have the properties stated in \Cref{thm:f(UV)-Ansatz}(a,b), and
  \item $\scrH_i(\cdot,\cdot)$ for $i=1,2$ and $f(U,V)$ are continuous.
\end{enumerate}
Hence both \Cref{thm:cvg4SCF4NPDo-GS:JSVD,thm:cvg4SCF4accNPDo-GS:JSVD,thm:cvg4SCF4NPDo-Jac:JSVD} remain valid for solving
\begin{equation}\label{eq:opt-on-STM2}
\max_{U\in\STM{k}{n_1},\,V\in\STM{k}{n_2}}  f(U,V)
\end{equation}
with a general objective $f$ that possesses the two properties listed above where
$$
\scrH_1(U,V):=\frac {\partial f(U,V)}{\partial U}, \quad
\scrH_2(U,V):=\frac {\partial f(U,V)}{\partial V}.
$$
With that in mind, we point out that it is straightforward to extend all, including the algorithm and convergence analysis, for solving an even more general
\begin{equation}\label{eq:opt-on-STMm}
\max_{U_i\in\STM{k}{n_i}, i=1,2,\ldots,m}  f(U_1,U_2,\ldots,U_m).
\end{equation}
Detail is omitted.

\section{Numerical Demonstration on \pjbsvd}\label{sec:egs-jbSVD}
To demonstrate the effectiveness of \Cref{alg:NPDo-pjbSVD,alg:accNPDo-pjbSVD} on  principal  Joint SVD-type block diagonalization
(\pjbsvd), we will conduct two experiments.
In both experiments, we first generate
two random orthogonal/unitary matrices $Q_i$, in MATLAB, by
\begin{subequations}\label{eq:aJSVD}
\begin{align}
&Q_1={\tt orth}({\tt randn}(n_1))
\quad\mbox{or}\quad
{\tt orth}({\tt randn}(n_1)+{\tt 1i*}{\tt randn}(n_1)), \label{eq:aJSVD-1} \\
&Q_2={\tt orth}({\tt randn}(n_2))
\quad\mbox{or}\quad
{\tt orth}({\tt randn}(n_2)+{\tt 1i*}{\tt randn}(n_2)), \label{eq:aJSVD-2}
\end{align}
depending on the choice of {\em real\/} or {\em complex\/}  matrices, and then loop over $\ell$ from $1$ to $N$ to execute
\begin{align}
C_{\ell}&={\tt randn}(n_1,n_2)
\quad\mbox{or}\quad
{\tt randn}(n_1,n_2)+{\tt 1i*}{\tt randn}(n_1,n_2), \label{eq:aJSVD-2a} \\
\mbox{either}\quad
D_{\ell}&=10\cdot\begin{bmatrix}
{\tt diag}({\tt randn}(n_2,1)) \\
           0
         \end{bmatrix} \quad\mbox{for \pjsvd\ or}\quad \label{eq:aJSVD-2b} \\
D_{\ell}&=10\cdot\begin{bmatrix}
{\tt diag}({\tt randn}(k_1),\ldots,{\tt randn}(k_t)) \\
           0
         \end{bmatrix}\quad\mbox{for \pjbsvd}, \label{eq:aJSVD-2c} \\
\intertext{and finally,}
B_{\ell}&=Q_1^{\HH}D_{\ell}Q_2+\eta C_{\ell}, \label{eq:aJSVD-2d}
\end{align}
where $\eta$ is an adjustable parameter, e.g., $10^{0}\sim 10^{-3}$, and $n_1=1.1n_2$.
All experiments are conducted within the MATLAB environment (MATLAB R2022) on a Dell Precision 3660 desktop with an
Intel i9 processor (3200 Mhz), 32 GB memory, running Microsoft Windows 11 Enterprise.
\end{subequations}
We will compare \Cref{alg:NPDo-pjbSVD} and its LOCG-accelerated version, named NPDo and accNPDo
along with suffixes ``-GS'' and ``-Jac'' to indicate their respective combinations with
either Gauss-Seidel-type updating or Jacobi-type updating. In total, for each set $\{B_{\ell}\}_{\ell=1}^N$,
four methods: NPDo-GS, NPDo-Jac, accNPDo-GS and accNPDo-Jac,
are used to compute its \pjbsvd.

\begin{figure}[th]
{\centering
\begin{tabular}{lcc}
  & \small NPDo-GS & \small NPDo-Jac \\
\rotatebox{90}{\hspace*{1.8cm}real} &
\resizebox*{0.41\textwidth}{0.18\textheight}{\includegraphics{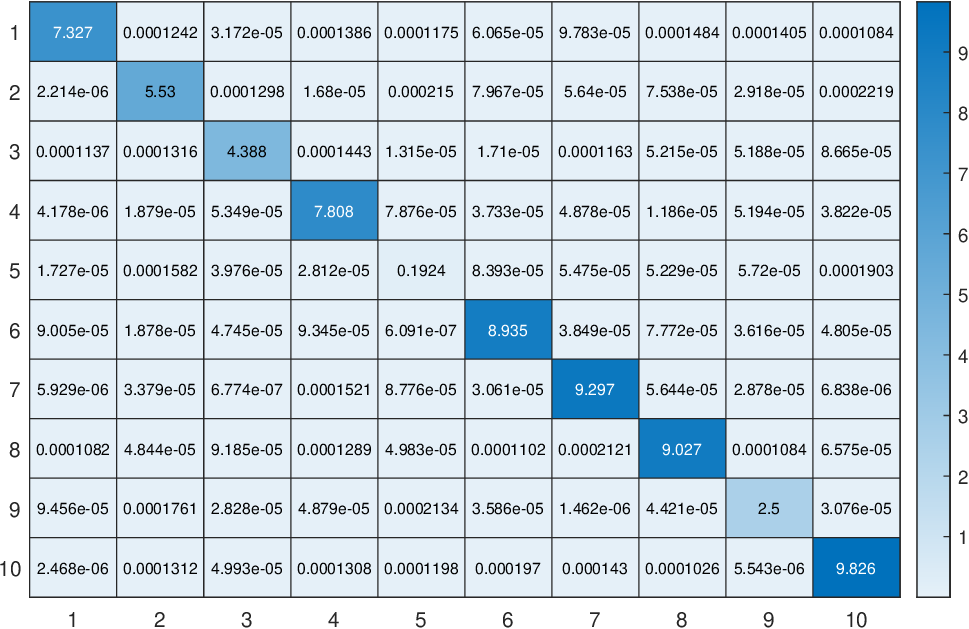}}
  & \resizebox*{0.41\textwidth}{0.18\textheight}{\includegraphics{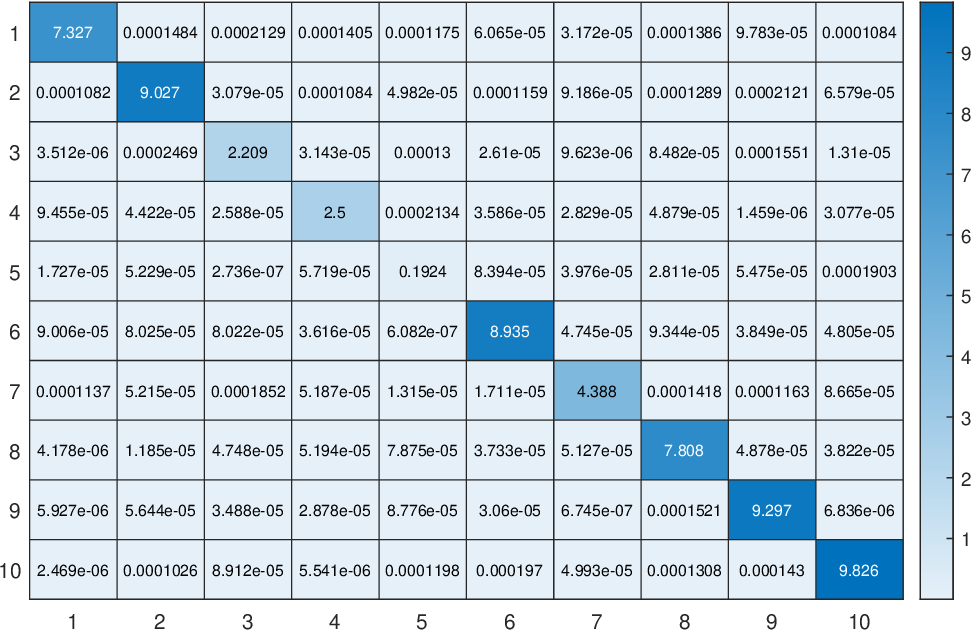}} \\
\rotatebox{90}{\hspace*{1.4cm}complex} &
\resizebox*{0.41\textwidth}{0.18\textheight}{\includegraphics{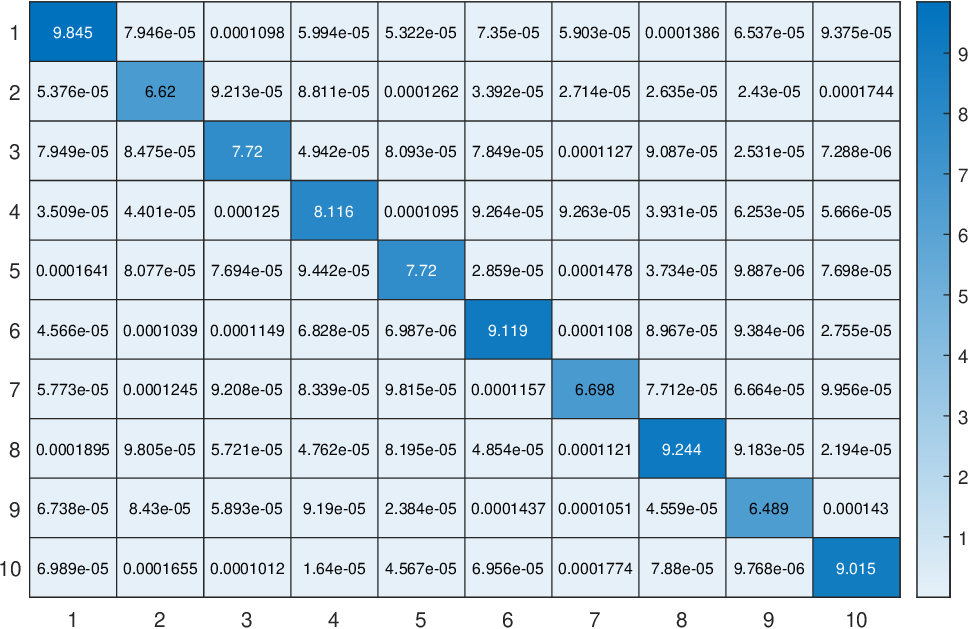}}
  & \resizebox*{0.41\textwidth}{0.18\textheight}{\includegraphics{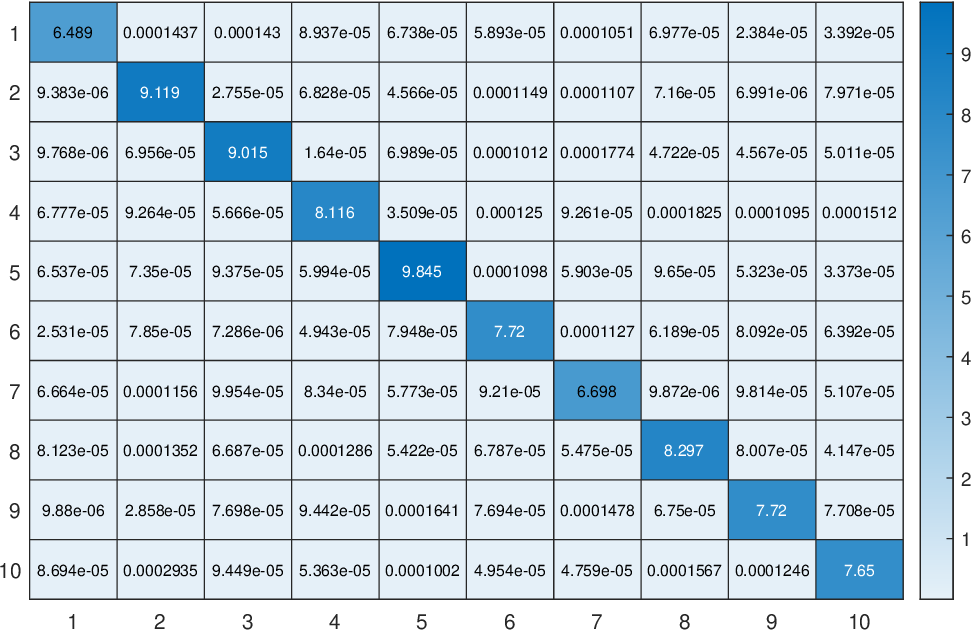}}
\end{tabular}\par
}
\vspace{-0.15 cm}
\caption{\small
    Heatmaps of the leading $10$-by-$10$ principal submatrices of converged $U^{\HH}B_1V$
    for two randomly generated sets $\{B_{\ell}\}_{\ell=1}^{10}$ according to \eqref{eq:aJSVD} with $\eta=10^{-4}$, by NPDo.
  }
\label{fig:heatmaps}
\end{figure}

Our first experiment is to use the heatmap to visually  demonstrate that the NPDo approach solves \eqref{eq:pbj-intro} to yield
\pjsvd. Let $n_2=500$ and $n_1=1.1n_2=550$, $N=10$, $k=10$, according to \eqref{eq:aJSVD} with $\eta=10^{-4}$, and generate two random sets $\{B_{\ell}\}_{\ell=1}^N$, real  or  complex. \Cref{fig:heatmaps} plots the heatmaps of $U^{\HH}B_1V$ by \Cref{alg:NPDo-pjbSVD}. It clearly shows that the ``mass'' of $B_1$ is moved principally to the diagonal of $U^{\HH}B_1V$
by both NPDo-GS and NPDo-Jac. We also examined the other $U^{\HH}B_{\ell}V$ and those by accNPDo and witnessed the same thing: large diagonal entries and very small off-diagonal entries.

\begin{figure}[t]
{\centering
\begin{tabular}{lcccc}
\rotatebox{90}{\hspace*{1.2cm}$\eta=10^0$} &
\resizebox*{0.21\textwidth}{0.14\textheight}{\includegraphics{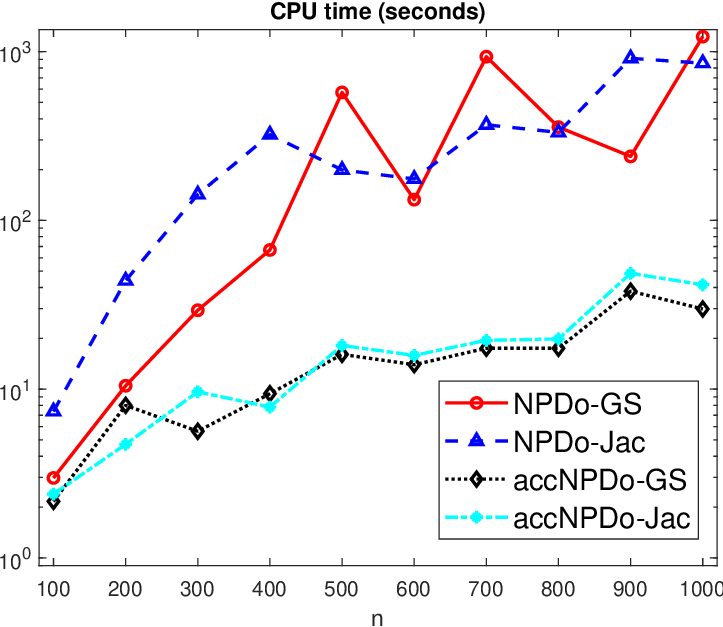}}
  & \resizebox*{0.21\textwidth}{0.14\textheight}{\includegraphics{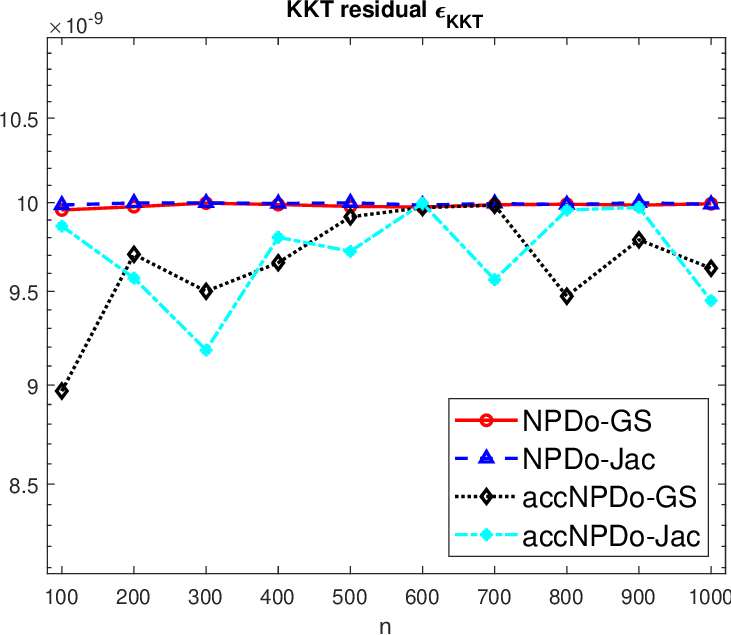}}
  & \resizebox*{0.21\textwidth}{0.14\textheight}{\includegraphics{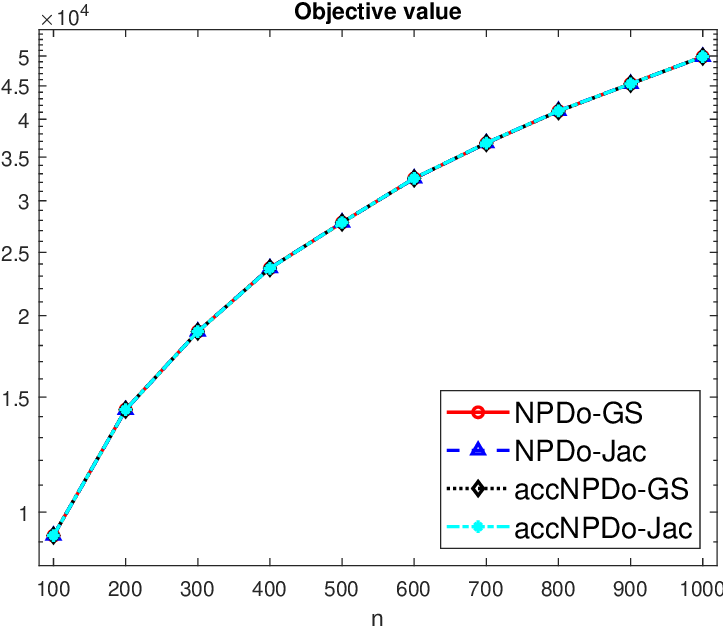}}
  & \resizebox*{0.21\textwidth}{0.14\textheight}{\includegraphics{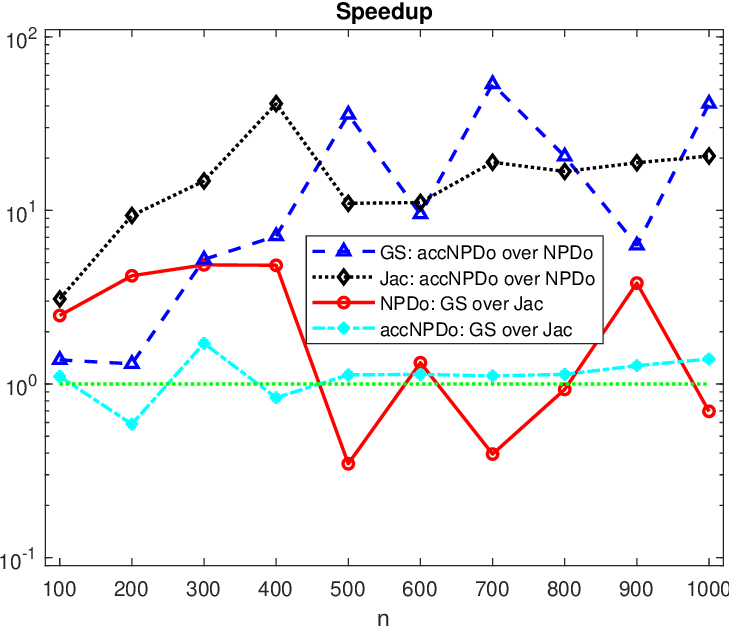}} \\
\rotatebox{90}{\hspace*{1.2cm}$\eta=10^{-1}$} &
\resizebox*{0.21\textwidth}{0.14\textheight}{\includegraphics{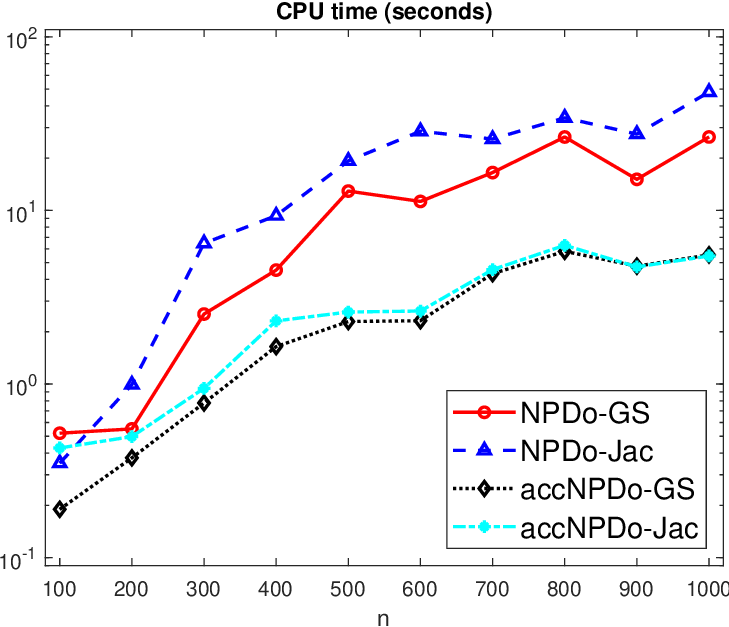}}
  & \resizebox*{0.21\textwidth}{0.14\textheight}{\includegraphics{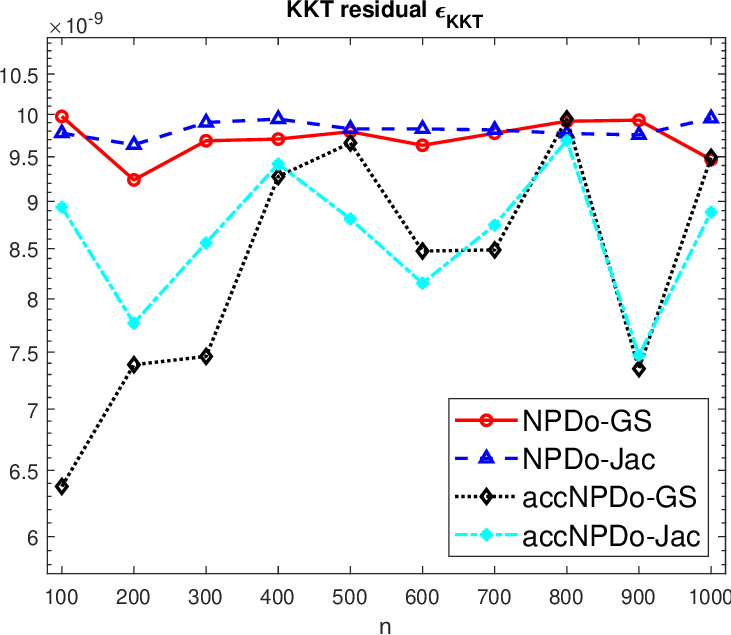}}
  & \resizebox*{0.21\textwidth}{0.14\textheight}{\includegraphics{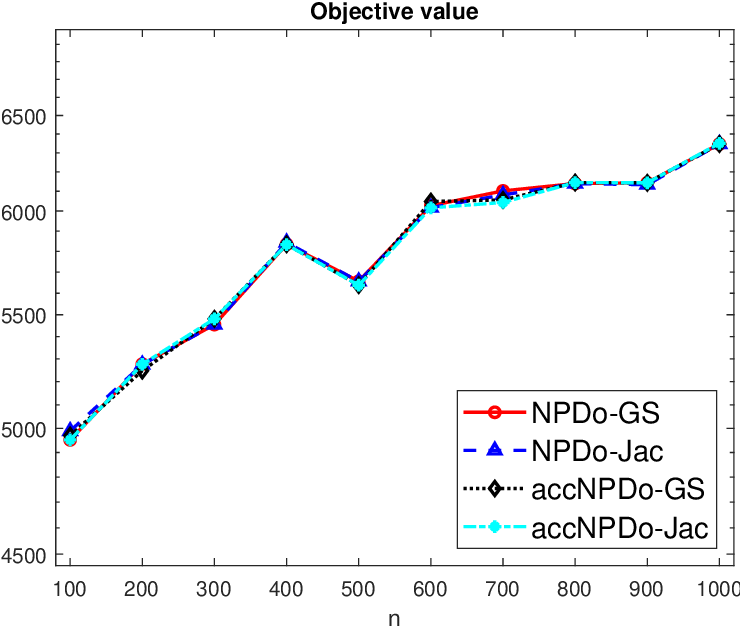}}
  & \resizebox*{0.21\textwidth}{0.14\textheight}{\includegraphics{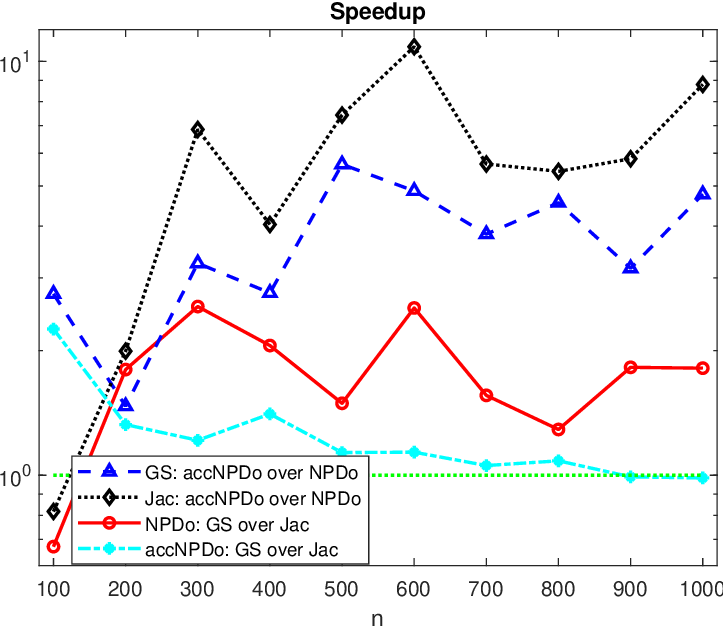}} \\
\rotatebox{90}{\hspace*{1.2cm}$\eta=10^{-2}$} &
\resizebox*{0.21\textwidth}{0.14\textheight}{\includegraphics{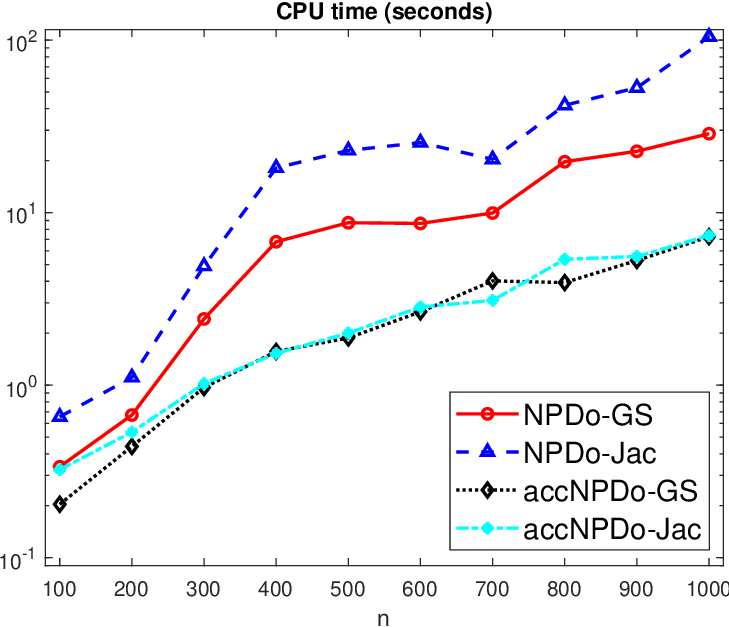}}
  & \resizebox*{0.21\textwidth}{0.14\textheight}{\includegraphics{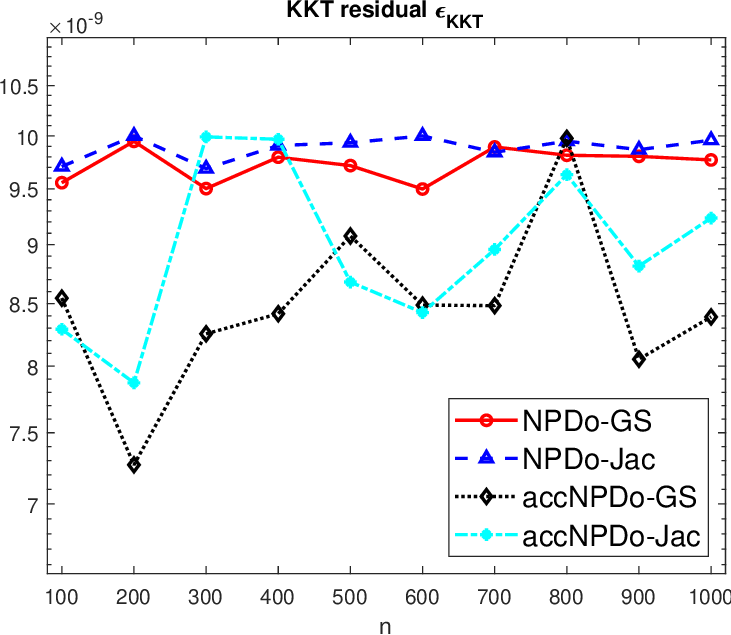}}
  & \resizebox*{0.21\textwidth}{0.14\textheight}{\includegraphics{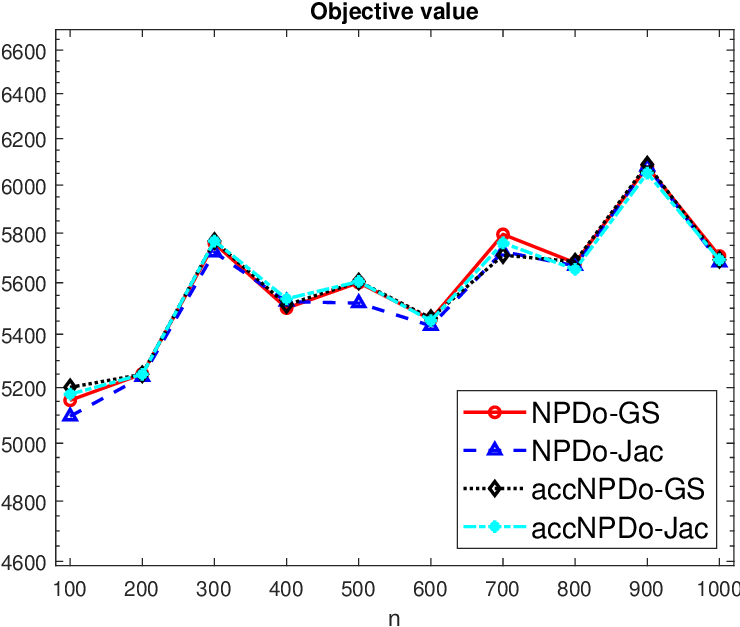}}
  & \resizebox*{0.21\textwidth}{0.14\textheight}{\includegraphics{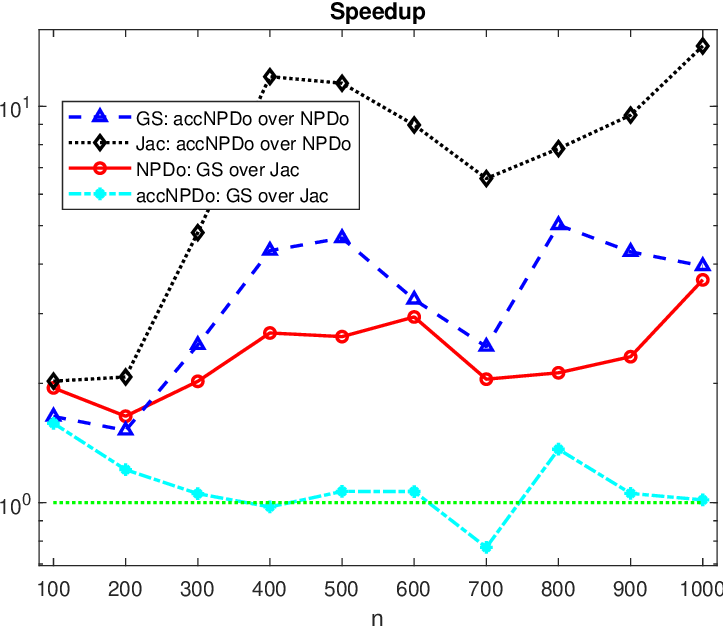}}  \\
\rotatebox{90}{\hspace*{1.2cm}$\eta=10^{-3}$} &
\resizebox*{0.21\textwidth}{0.14\textheight}{\includegraphics{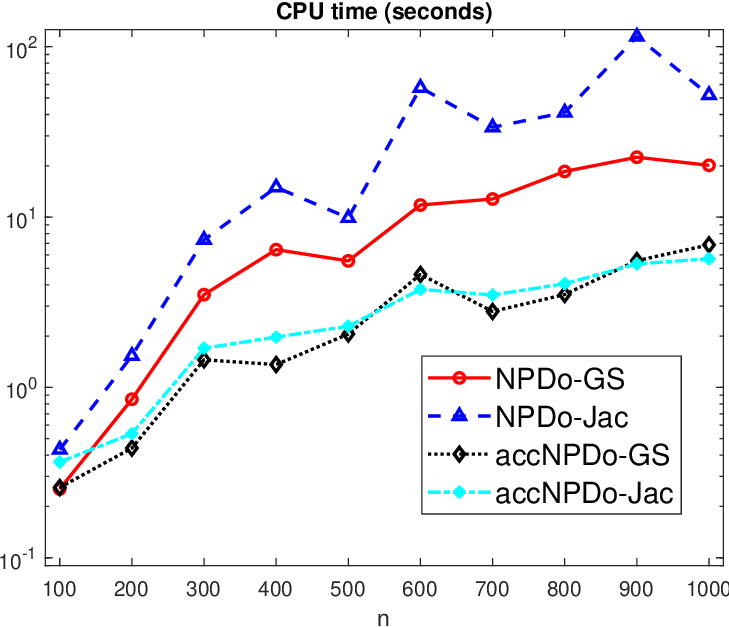}}
  & \resizebox*{0.21\textwidth}{0.14\textheight}{\includegraphics{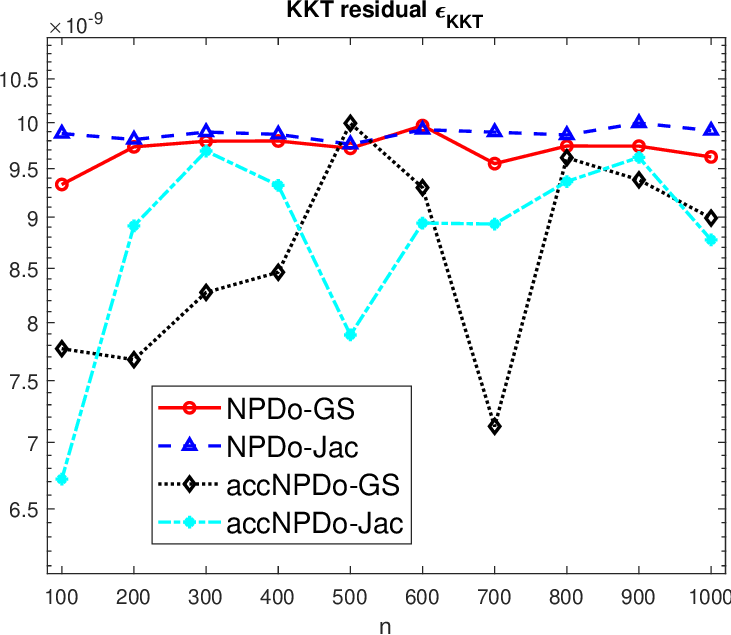}}
  & \resizebox*{0.21\textwidth}{0.14\textheight}{\includegraphics{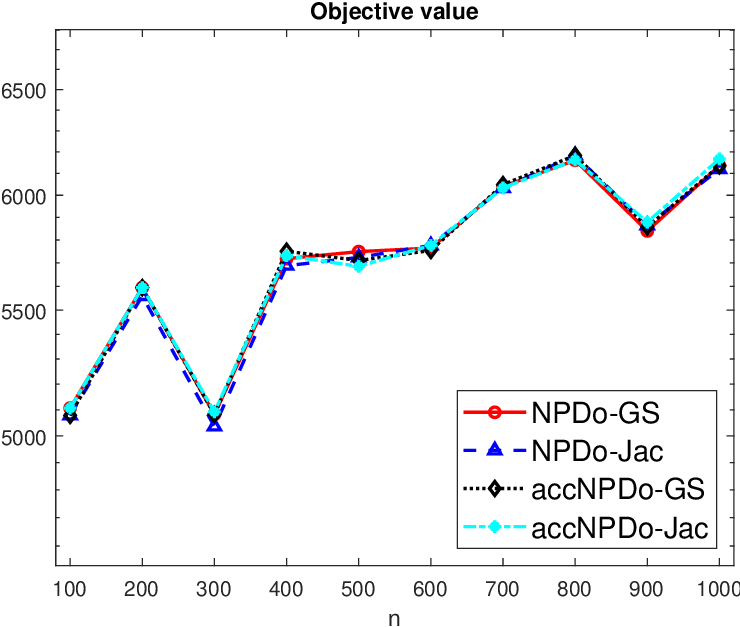}}
  & \resizebox*{0.21\textwidth}{0.14\textheight}{\includegraphics{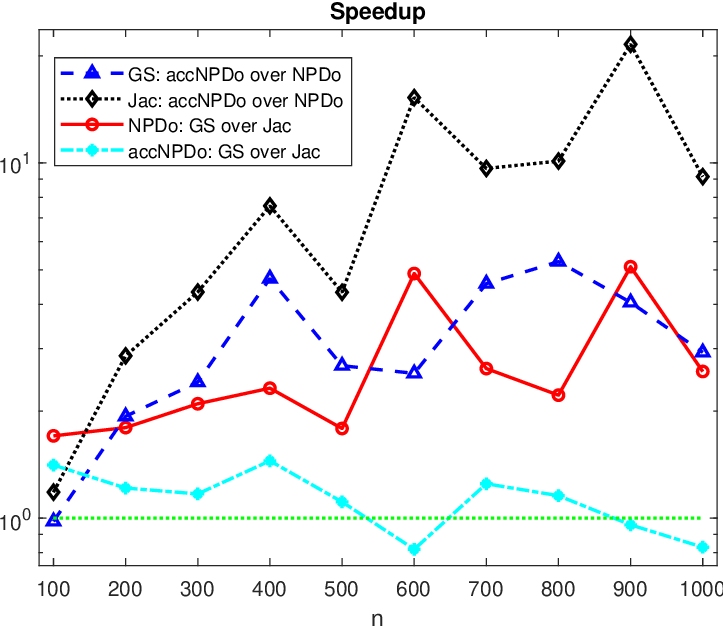}}
\end{tabular}\par
}
\vspace{-0.15 cm}
\caption{\small
   Performance for \pjsvd\ by NPDo and accNPDo combined with either Gauss-Seidel-type and Jacobi-type updating
   on {\em complex\/} $B_{\ell}$ generated according to
   \eqref{eq:aJSVD}, where $k=10$, $N=10$, and $n_2$ varies from $10^2$ to $10^3$ and $n_1=1.1n_2$.
  }
\label{fig:aJSVD-n-vary}
\end{figure}

Next in \Cref{fig:aJSVD-n-vary,fig:aJBSVD-n-vary}, we plot
CPU time, KKT residual $\epsilon_{\KKT}$, objective value, and speedup as $n_2$ varies for $10^2$ to $10^3$ while
$\eta$ varies from $10^0$ to $10^{-3}$. We have the following observations:
\begin{enumerate}[(1)]
  \item KKT residuals $\epsilon_{\KKT}$ are no bigger than $10^{-8}$, indicating convergence, while the ones by accNPDo with both GS and Jac
        are overwhelmingly the smallest, most of the time.
  \item The objective values are comparable but not the same for all, indicating convergence to different
        stationary points.
  \item For the same updating scheme: GS or Jac, accNPDo can be several times faster than NPDo,
        except possibly for small $n_2=100$ and $200$. The speed gain is more pronounced for \jsvd\ than for \jbsvd.
  \item For \jsvd, NPDo-GS can be several times faster than NPDo-Jac, but for \jbsvd, it does not appear that one has advantage
        over the other.
  \item Despite less satisfactory convergence results in \Cref{thm:cvg4SCF4NPDo-Jac:JSVD} for NPDo-Jac
        than those in \Cref{thm:cvg4SCF4NPDo-GS:JSVD} for NPDo-GS, Alternating SCF (ASCF) with Jacobi-type updating works surprisingly well.
\end{enumerate}

\begin{figure}[t]
{\centering
\begin{tabular}{lcccc}
\rotatebox{90}{\hspace*{1.2cm}$\eta=10^0$} &
\resizebox*{0.21\textwidth}{0.14\textheight}{\includegraphics{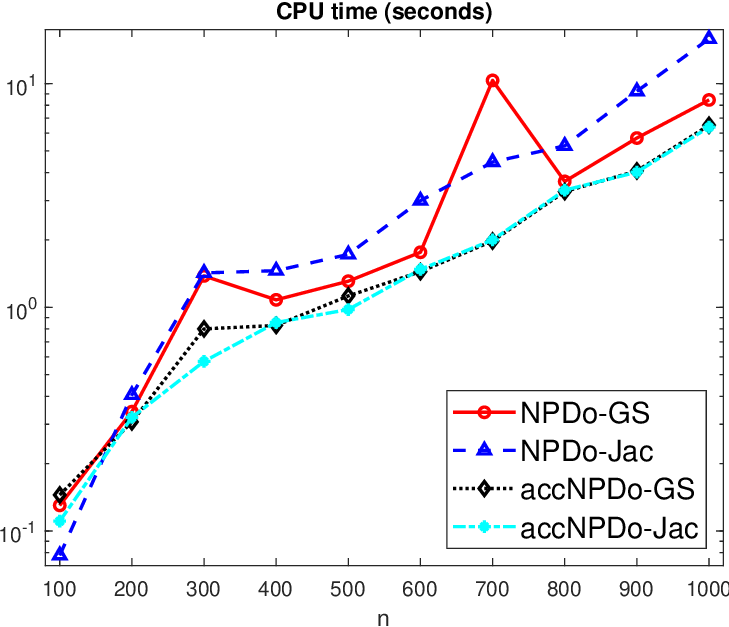}}
  & \resizebox*{0.21\textwidth}{0.14\textheight}{\includegraphics{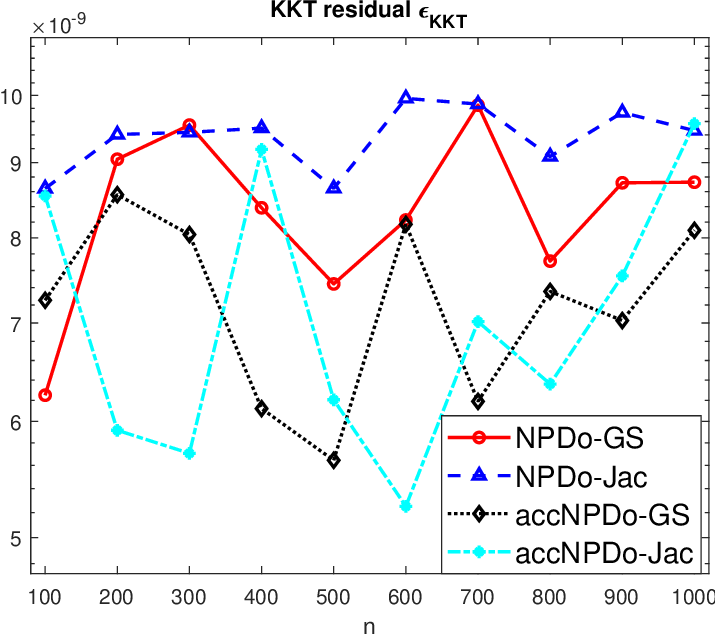}}
  & \resizebox*{0.21\textwidth}{0.14\textheight}{\includegraphics{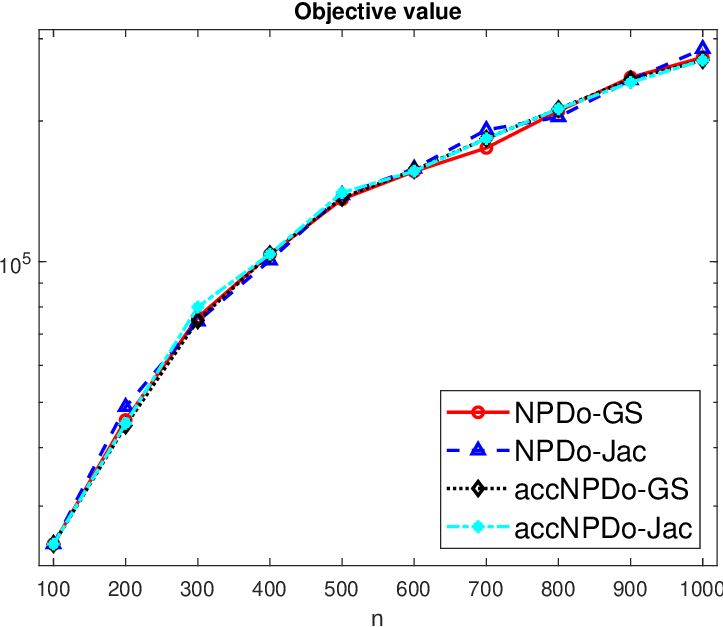}}
  & \resizebox*{0.21\textwidth}{0.14\textheight}{\includegraphics{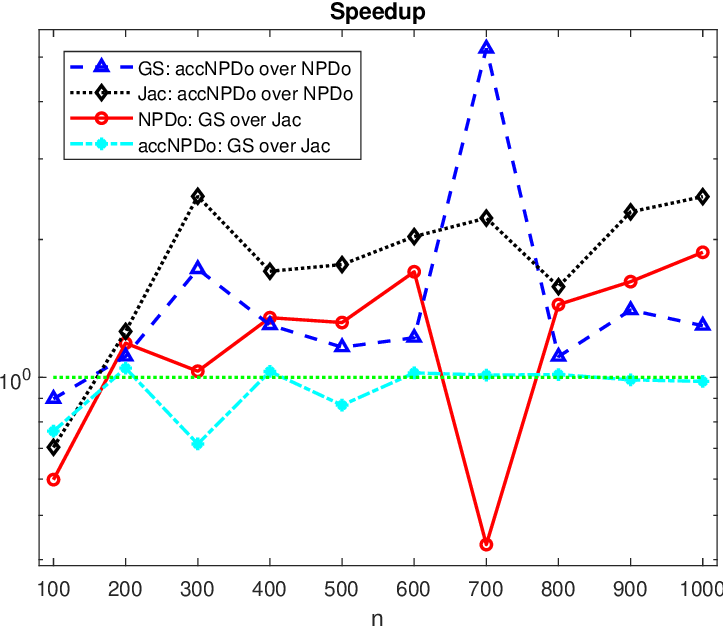}} \\
\rotatebox{90}{\hspace*{1.2cm}$\eta=10^{-1}$} &
\resizebox*{0.21\textwidth}{0.14\textheight}{\includegraphics{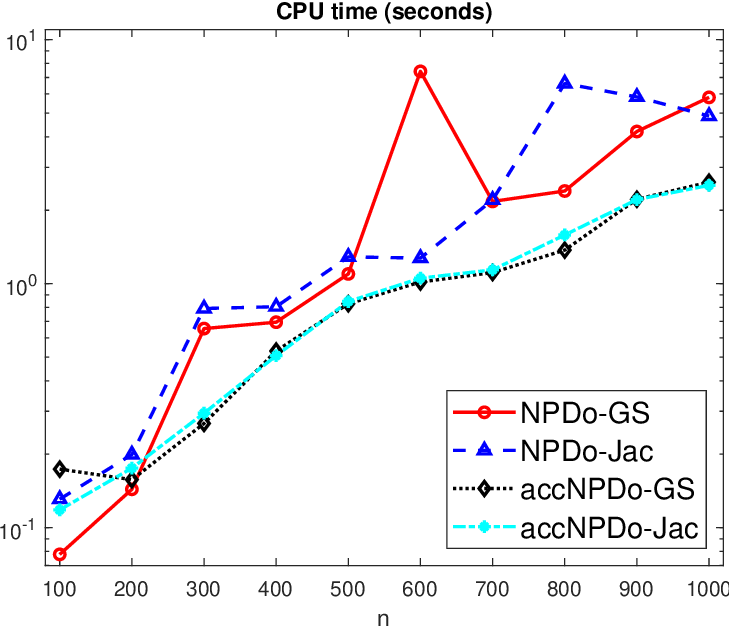}}
  & \resizebox*{0.21\textwidth}{0.14\textheight}{\includegraphics{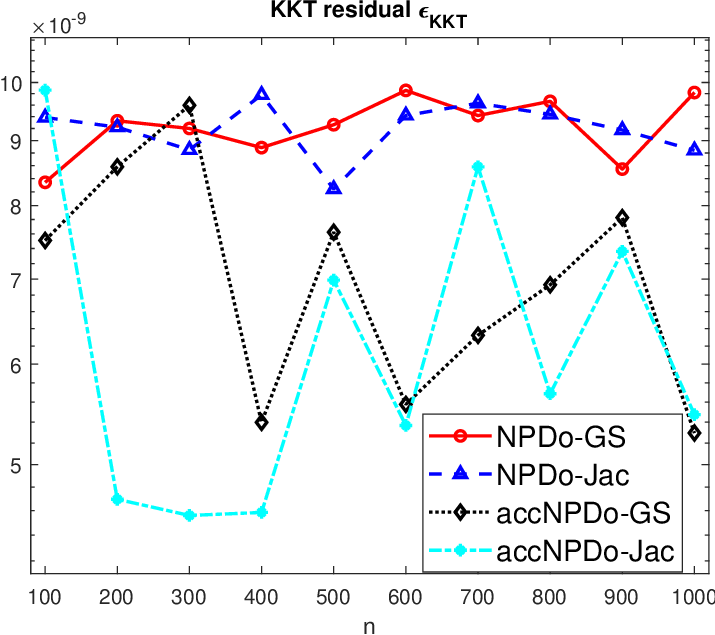}}
  & \resizebox*{0.21\textwidth}{0.14\textheight}{\includegraphics{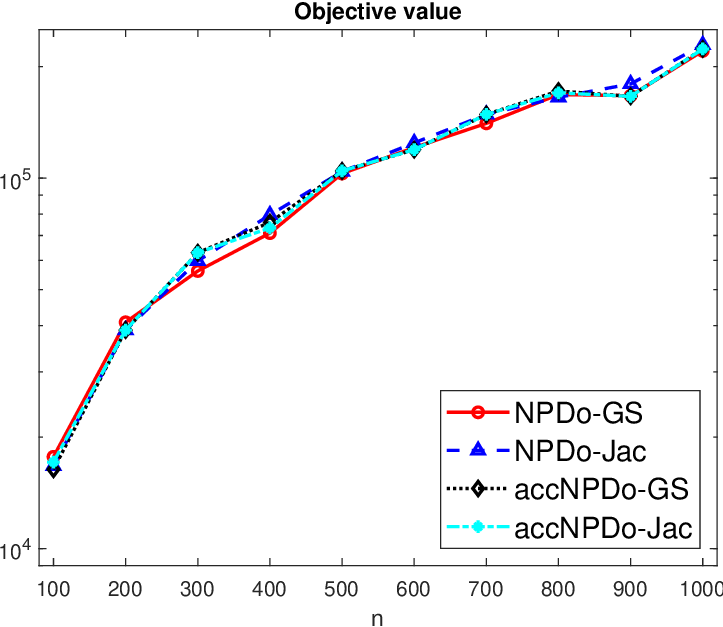}}
  & \resizebox*{0.21\textwidth}{0.14\textheight}{\includegraphics{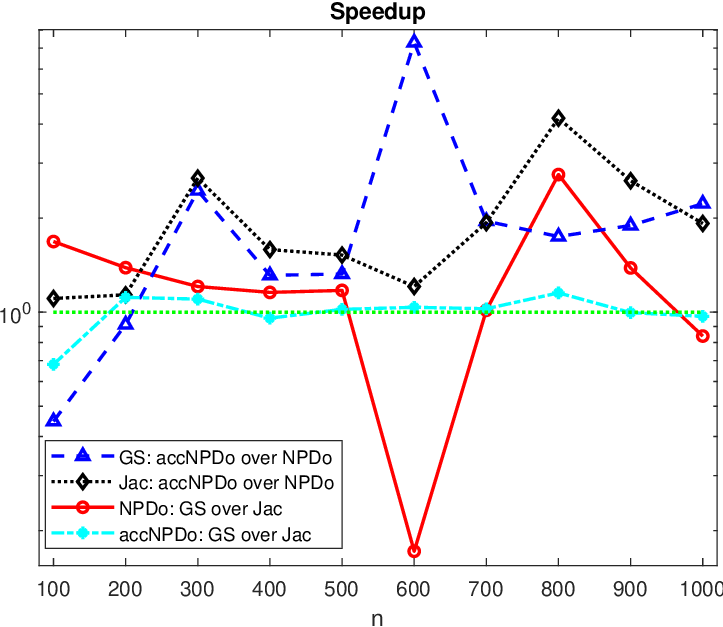}} \\
\rotatebox{90}{\hspace*{1.2cm}$\eta=10^{-2}$} &
\resizebox*{0.21\textwidth}{0.14\textheight}{\includegraphics{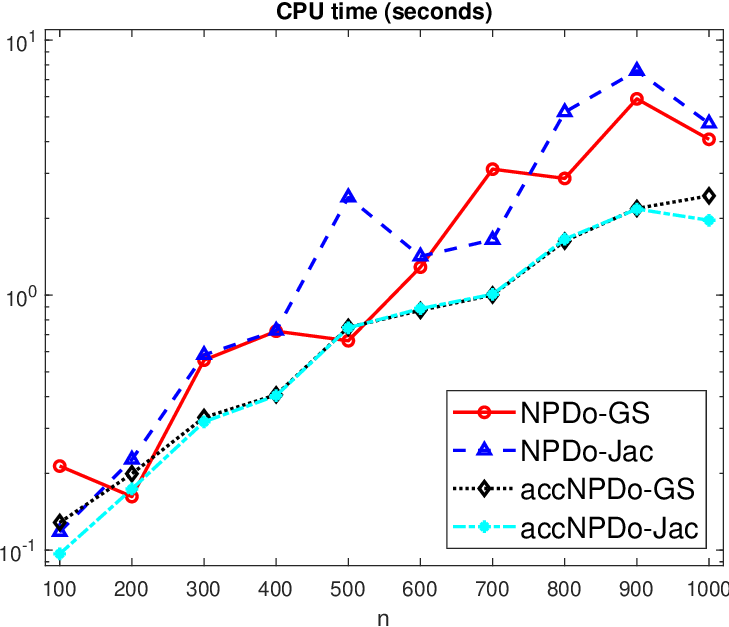}}
  & \resizebox*{0.21\textwidth}{0.14\textheight}{\includegraphics{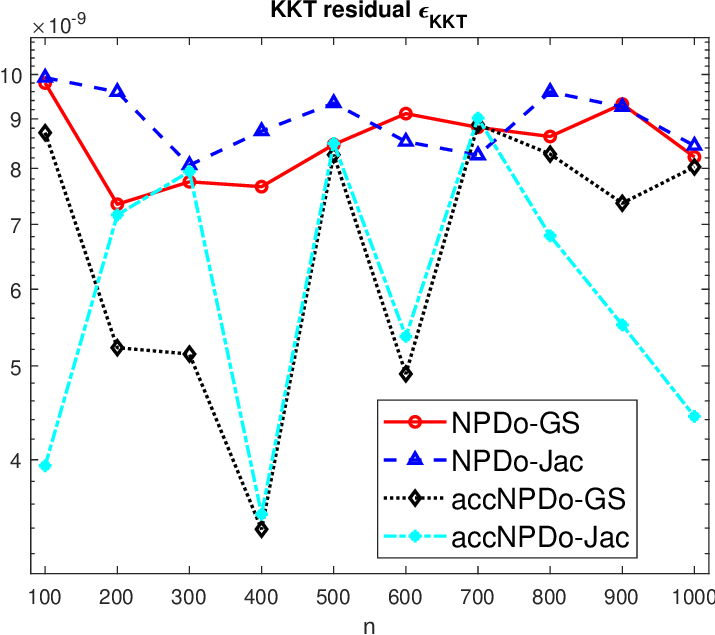}}
  & \resizebox*{0.21\textwidth}{0.14\textheight}{\includegraphics{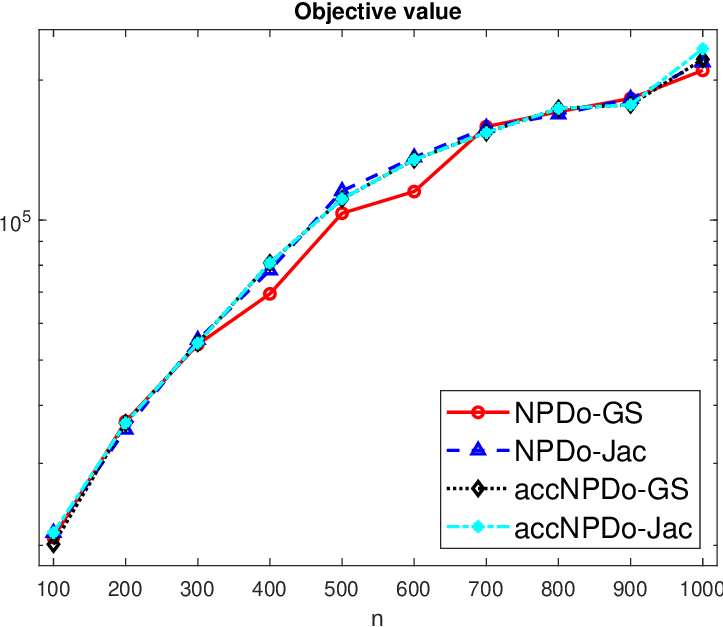}}
  & \resizebox*{0.21\textwidth}{0.14\textheight}{\includegraphics{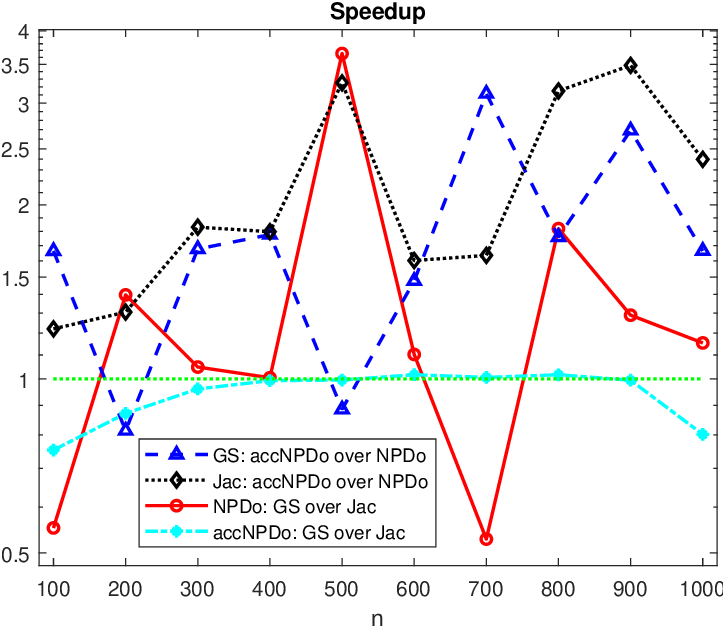}}  \\
\rotatebox{90}{\hspace*{1.2cm}$\eta=10^{-3}$} &
\resizebox*{0.21\textwidth}{0.14\textheight}{\includegraphics{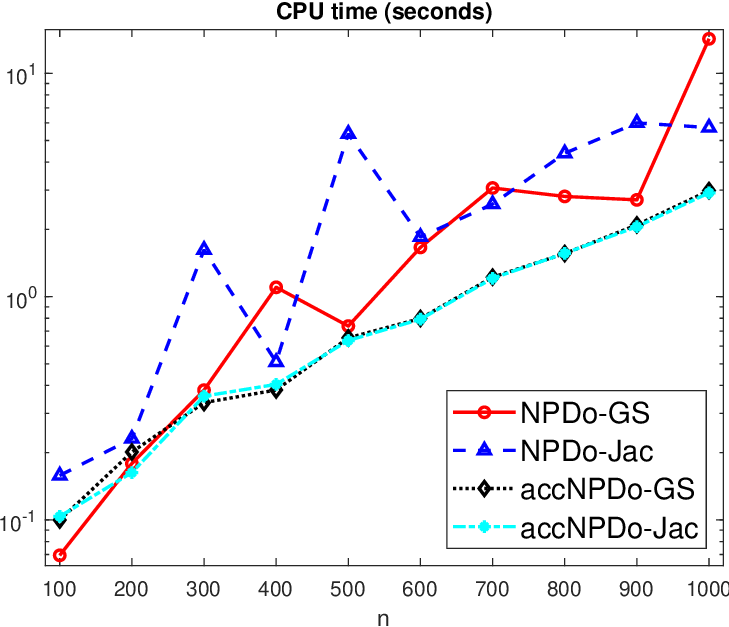}}
  & \resizebox*{0.21\textwidth}{0.14\textheight}{\includegraphics{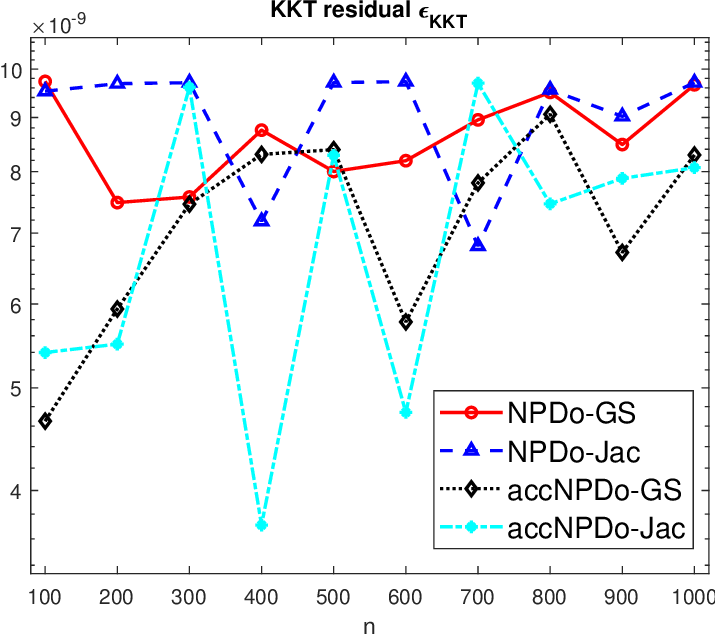}}
  & \resizebox*{0.21\textwidth}{0.14\textheight}{\includegraphics{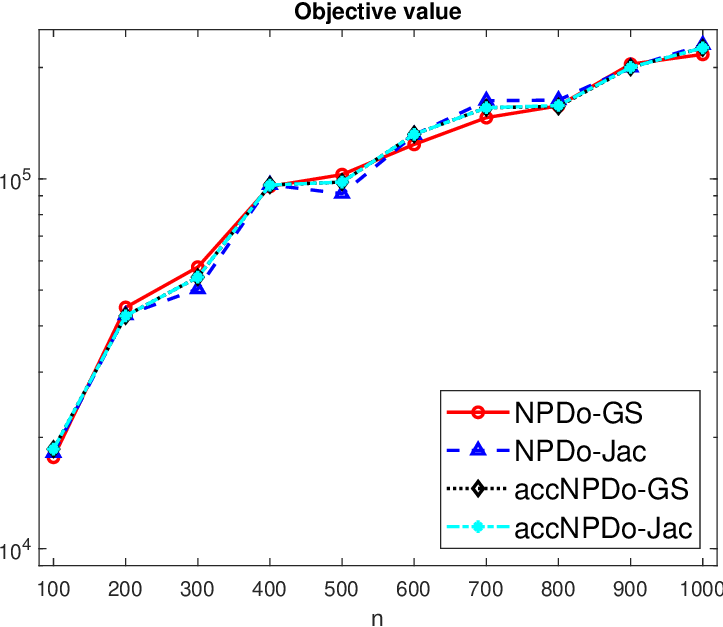}}
  & \resizebox*{0.21\textwidth}{0.14\textheight}{\includegraphics{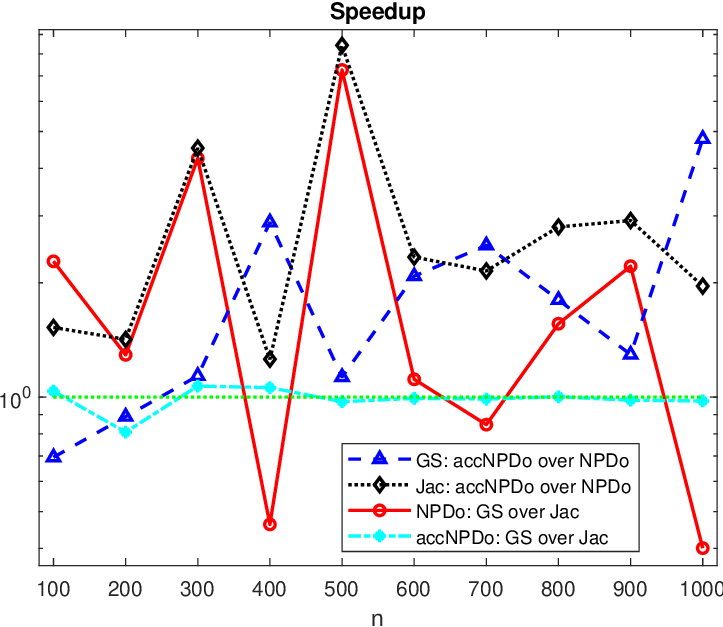}}
\end{tabular}\par
}
\vspace{-0.15 cm}
\caption{\small
   Performance for \pjbsvd\ by NPDo and accNPDo combined with either Gauss-Seidel-type and Jacobi-type updating
   on {\em complex\/} $B_{\ell}$ generated according to
   \eqref{eq:aJSVD}, where  each $k_i=2$, $t=5$, $k=2t=10$, $N=10$, and $n$ varies from $10^2$ to $10^3$ and $n_1=1.1n_2$.
  }
\label{fig:aJBSVD-n-vary}
\end{figure}

\section{Concluding Remarks}\label{sec:concl}
We are interested in the principal joint SVD-type block-diagonalization (\pjbsvd) of several matrices, which becomes
the principal joint SVD-type diagonalization (\pjsvd) when each block size is 1-by-1.
Specifically, given $B_{\ell}\in\bbC^{n_1\times n_2}$ for $1\le \ell\le N$, we seek two orthonormal matrices $U\in\bbC^{n_1\times k}$ and $V\in\bbC^{n_2\times k}$ such that all $U^{\HH}B_{\ell}V$ are block-diagonal with the same given block-diagonal structure in the dominant way.
We propose an NPDo (nonlinear polar decomposition with orthonormal polar factor dependency)
approach for that purpose. It includes an alternating self-consistent-field (ASCF) iteration to
numerically solve the equations from the KKT condition and its comvergence theory.
It is shown the SCF iteration combined with Gauss-Seidel-type updating is globally convergent to a stationary point while the objective increases monotonically. Convergence of ASCF with Jacobi-type updating is also analyzed.
An accelerated version of the approach via the locally optimal conjugate gradient (LOCG) technique is also established.
Numerical experiments are presented to illustrate the efficiency of the NPDo approach.
Our focus in this paper is on algorithmic development and we refer the reader to \cite{copp:2010} for
an application of \pjsvd.

Our joint SVD-type block-diagonalization so far enforces two things: 1)  both $U$ and $V$ have the same number of columns, i.e., $k$ of them,
and 2) both column-partitioned in the same way as in \eqref{eq:P-part'n} (and thus each block is of square shape).
We remark that these two requirements are not necessary as far as the theory and
numerical methods are concerned, but made so because that is what is likely the case in practice. In fact,
our theory and numerical methods are equally valid for more general case: $U\in\STM{k}{n_1}$,   $V\in\STM{k'}{n_2}$,
each $U^{\HH}B_{\ell}V$ partitioned
in the form
$$
\kbordermatrix{ &\sss k_1' & \sss k_2' & \sss \cdots &\sss k_t' \\
         \sss k_1 & A_{11} & A_{12} & \cdots & A_{1t} \\
         \sss k_2 & A_{21} & A_{22} & \cdots & A_{2t} \\
         \sss \vdots & \vdots & \vdots &  & \vdots \\
         \sss k_t & A_{t1} & A_{t2} & \cdots & A_{tt} }
$$
according to partition $\tau_k$ of $k$ in \eqref{eq:tau-n} and partition $\tau_{k'}=(k_1',\dots,k_t')$ of $k'$,
and correspondingly
$$
\BDiag_{(\tau_k,\tau_{k'})}(A)=\diag(A_{11},\dots,A_{tt}).
$$
Our earlier setting is a special case: $k=k'$ and $\tau_k=\tau_{k'}$. Whether this more general setting has a practical application
remains to be seen.

{\small
\def\noopsort#1{}\def\l{\char32l}\def\v#1{{\accent20 #1}}
  \let\^^_=\v\def\hbk{hardback}\def\pbk{paperback}

}

\end{document}